\documentclass[10pt]{article}
\usepackage{amssymb}
\usepackage{graphicx}
\usepackage{xcolor} 
\usepackage{tensor}
\usepackage{fullpage} 
\usepackage{amsmath}
\usepackage{amsthm}
\usepackage{verbatim}
\usepackage{hyperref}
\usepackage{enumitem}
\setlist[enumerate]{leftmargin=1.2em}
\setlist[itemize]{leftmargin=1.2em}

\usepackage{seqsplit,mathtools}

\definecolor{green}{rgb}{0,0.8,0} 

\newcommand{\Red}[1]{\begingroup\color{red} #1\endgroup} 
\newcommand{\Blue}[1]{\begingroup\color{blue} #1\endgroup} 

\newtheorem{theorem}{Theorem}[section]

\newtheorem{lemma}[theorem]{Lemma}
\newtheorem{proposition}[theorem]{Proposition}

\theoremstyle{definition}

\theoremstyle{remark}
\newtheorem{remark}[theorem]{Remark}

\numberwithin{equation}{section}
\newcommand{\nrm}[1]{\Vert#1\Vert}

\newcommand{\brk}[1]{\langle#1\rangle}

\newcommand{\nnrm}[1]{{\vert\kern-0.25ex\vert\kern-0.25ex\vert #1 
		\vert\kern-0.25ex\vert\kern-0.25ex\vert}}

\newcommand{\dist}{\mathrm{dist}\,}

\newcommand{\supp}{{\mathrm{supp}}\,}

\renewcommand{\Re}{\mathrm{Re}}

\newcommand{\lap}{\Delta}

\newcommand{\rd}{\partial}
\newcommand{\nb}{\nabla}

\newcommand{\alp}{\alpha}
\newcommand{\bt}{\beta}

\newcommand{\dlt}{\delta}

\newcommand{\varep}{\varepsilon}

\newcommand{\lmb}{\lambda}

\newcommand{\tht}{\theta}
\newcommand{\Tht}{\Theta}

\newcommand{\omg}{\omega}
\newcommand{\Omg}{\Omega}

\newcommand{\zt}{\zeta}

\newcommand{\bbR}{\mathbb R}

\newcommand{\bbZ}{\mathbb Z}

\newcommand{\calA}{\mathcal A}

\newcommand{\calE}{\mathcal E}

\newcommand{\calL}{\mathcal L}
\newcommand{\calM}{\mathcal M}
\newcommand{\calN}{\mathcal N}
\newcommand{\calO}{\mathcal O}

\begin{document}
\bibliographystyle{plain}
 \title{Stability and instability of Kelvin waves} 
\author{Kyudong Choi\thanks{Department of Mathematical Sciences, Ulsan National Institute of Science and Technology, 50 UNIST-gil, Eonyang-eup, Ulju-gun, Ulsan, Republic of Korea. Email: kchoi@unist.ac.kr.} 
	\and In-Jee Jeong\thanks{Department of Mathematical Sciences and RIM, Seoul National University, 1 Gwanak-ro, Gwanak-gu, Seoul 08826, Republic of Korea. Email: injee\_j@snu.ac.kr.}
}
\date\today
 
  \maketitle 
 
\begin{abstract} 
	The $m$-waves of Kelvin are uniformly rotating patch solutions of the 2D Euler equations with $m$-fold rotational symmetry for $m\ge2$. For Kelvin waves sufficiently close to the disc, we prove a nonlinear stability result up to an arbitrarily long time in the $L^1$ norm of the vorticity, for $m$-fold symmetric perturbations. {To obtain this result, we first prove that} the Kelvin wave is a strict local maximizer of the energy functional in some admissible class of patches, {which had been claimed by Wan in 1986. This gives an orbital stability result with a support condition on the evolution of perturbations, but using a Lagrangian bootstrap argument which traces the particle trajectories of the perturbation, we are able to drop the condition on the evolution. Based on this unconditional stability result,} we establish that long time filamentation, or formation of long arms, occurs near  the Kelvin waves, which have been observed in various numerical simulations. Additionally, we discuss stability of annular patches in the same variational framework.
\end{abstract}

\tableofcontents

\section{Introduction} 
{A few coherent vortices have been found in the two dimensional Euler equations, such as (rotating) disks, (sliding) dipoles, etc. The study of their stability (and instability) has been a classical topic in fluid dynamics it is believed to be relevant for long time behavior of high Reynolds number flows. Among them, we revisit} the  $m$-waves of Kelvin, which are uniformly rotating patch solutions of the two-dimensional incompressible Euler equations on $\bbR^2$ {in the vorticity form:}
\begin{equation}\label{eq:Euler}
	\left\{ 
	\begin{aligned}
		\rd_t\omg+u\cdot\nb\omg = 0, &	\\
		u = -\nb^\perp (-\lap)^{-1}\omg. & 
	\end{aligned}
	\right. 
\end{equation} Here, $\omg(t,\cdot) :\bbR^2\to \bbR$ and $u(t,\cdot) : \bbR^2\to\bbR^2$ denote the vorticity and velocity of the fluid at time $t$, respectively.  {For any integer $m\ge2$ and real $r_0>0$,}  the Kelvin waves can be parametrized by $\bt>0$ (see \cite{Burbea1982}); for a sufficiently small $\bt$, we shall write \begin{equation}\label{defn_v_sta}
	\begin{split}
		\omg^{m,\bt} = \mathbf{1}_{A^{m,\bt}}, \quad A^{m,\bt} = \left\{ (r,\tht): r  {<}
		 r_{0} + g^{m,\bt}(\tht) \right\}
	\end{split}
\end{equation} as the {$m$-wave of Kelvin with parameter $\bt$}, characterized by the property that  \begin{equation*}\label{eq:boundary-asym}
	\begin{split}
		g^{m,\bt}(\tht) = \bt\cos(m\tht) +  {o(\bt)} ,
	\end{split}
\end{equation*} where the $ {o(\bt)}$--term consists of expressions $\cos(km\tht)$ with $k>1$. {It turns out that for $\bt$ small, the function $g^{m,\bt}$ can be chosen uniquely in a way that \begin{equation}\label{act_v_state}
	\omega^{m,\beta}_{\Omega^{m,\beta}t}(x),\quad\,t\geq 0, \,x\in\mathbb{R}^2 
\end{equation} defines a solution of the Euler equations \eqref{eq:Euler} for some $\Omega^{m,\beta}\in\mathbb{R}$, which is the angular velocity of the Kelvin wave $\omg^{m,\bt}$. Here, we are using the notation $\omega_\alpha :=\omega (R_{-\alp}x)$ where $R_{\alp}$ is the counterclockwise rotation matrix by angle $\alp$ with respect to the origin.} In the rest of the introduction, we fix the reference length $r_0$ to be $1$ for simplicity, and take $B$ to be the open ball centered at the origin with radius $1$. 

Kirchhoff has discovered that ellipses define uniformly rotating patch solutions for any aspect ratio \cite{Kirchhoff}, which correspond to the case $m = 2$ in the above. For $m\ge3$, the existence of $m$-fold symmetric rotating patches bifurcating from the disc was first hinted by Kelvin in 1880  {(see Lamb \cite[231p]{Lamb}),} who computed that an infinitesimal perturbation of the disc with period $m$ rotates with the angular speed $1/2- 1/2m$. Then, an argument of existence was given by Burbea \cite{Burbea1982}  {in 1982} (who coined the term ``$m$-waves of Kelvin'') and then rigorously by  {Hmidi, Mateu, and Verdera in} \cite{HMV}  {only in 2013. We also refer to} 
the work \cite{HMW} of {Hassainia, Masmoudi, and Wheeler}
 where the authors study the behavior of the whole branch of solutions. Thanks to these works, we know that the boundary of these rotating patches are $C^\infty$--smooth and even real analytic.

In this paper, we focus on  nonlinear stability and instability of the Kelvin waves close to the disc, and  obtain the following results:\\



{\textit{Long time stability} (Theorem \ref{cor:finite_time}). {For any $T>0$, sufficiently localized  {$m$-fold symmetric} perturbations of the Kelvin wave stay close to the rotating Kelvin wave \eqref{act_v_state} in the time interval $[0,T]$.} \\
	
\textit{Instability: large perimeter growth} (Theorem \ref{thm:instability}). For any $M>0$, there exists an $L^1$-small patch perturbation of the Kelvin wave whose perimeter grows to become larger than $M$ in finite time.} \\


{The precise statements will be given in \S\ref{subsec:main}, but see Figure \ref{fig:fila} which illustrates both stable and {unstable} behavior of a 3-fold rotating state. While the ``bulk'' of the patch seems to converge to a Kelvin wave, long arms are constantly growing in time. This type of \textit{filamentation} instability can be generically observed for vortex patches \cite{Carrillo2000,Kamm,HZTM,DeZa,ZHR,Drit}.
}

\begin{figure}
	\centering
	\includegraphics[scale=0.23]{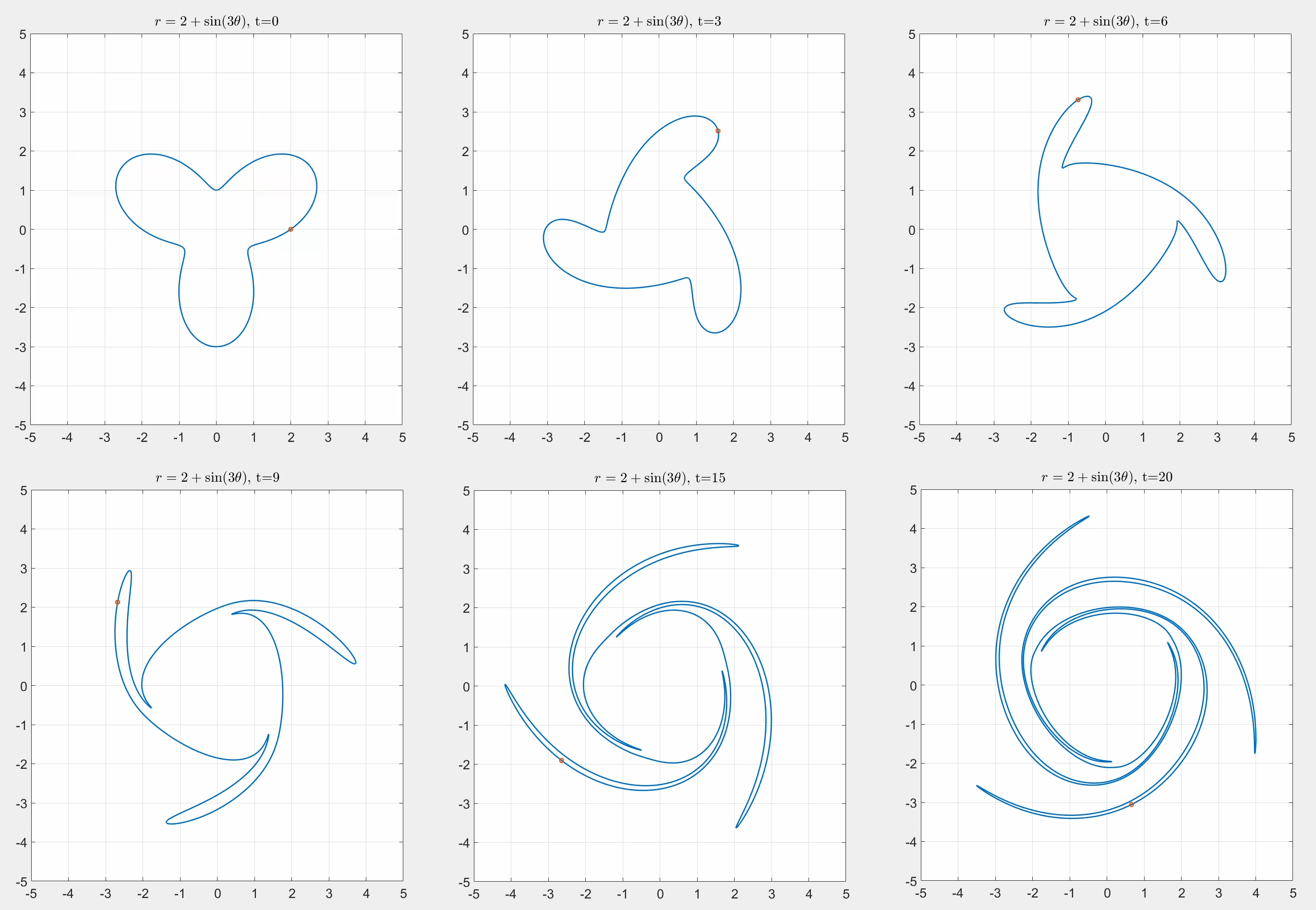}   
	\caption{Evolution of the patch initially defined by the region $\left\{ (r,\tht) : r < 2 + \sin(3\tht) \right\}$ at time moments $t = 0, 3, 6, 9,  15, 20$. Courtesy of Junho Choi.} \label{fig:fila}
\end{figure}

\subsection{Main results} \label{subsec:main}

Let us recall that  {$\omega^{m,\beta}_{\Omega^{m,\beta}t}(x)$ in \eqref{act_v_state} is the  rotating} solution of Euler with $\omg^{m,\bt}$ as the initial data. Our first result shows that localized and $L^1$ small perturbations of the Kelvin wave  {(for sufficiently small $\beta$)} stay $L^1$ close for an arbitrarily long time.  {Roughly speaking, the core part of the perturbed patch solutions is equal to that of  the   rotating Kelvin wave \eqref{act_v_state}.}

\begin{theorem}[Stability] \label{cor:finite_time}   
	Let $m\geq 2$ be an integer.  There are $\beta_1>0$ and 
	 {$r'>1$}
	 such that for any $\beta\in(0,\beta_1]$, $\varepsilon>0$, and $T>0$, there exists $\delta=\dlt(m,\bt,\varep,T)>0$ such that 
	if $\omega_0= \mathbf{1}_{A_0}$ for some $m$-fold symmetric open set $A_0\subset\mathbb{R}^2$ satisfying 
	\begin{equation}\label{ass_st_cor}
		\|\omega_0-\omega^{m,\beta}\|_{L^1(\mathbb{R}^2)}\leq\delta \qquad \mbox{and} \qquad 
	 {A_0\subset B_{r'}	}
	\end{equation} 
	then the solution $\omg(t)$ to \eqref{eq:Euler} with initial data $\omg_0$ satisfies 
	\begin{equation}\label{est_cor}
		\sup_{t\in[0,T]}\|\omega(t)-\omega_{\Omega^{m,\beta}t}^{m,\beta}\|_{L^1(\mathbb{R}^2)}\leq\varepsilon. 
	\end{equation}
	 {Here, $B_r,\,r>0$ is the  open ball centered at the origin with radius $r$, and   a set $A\subset \mathbb{R}^2$ is $m$-fold symmetric if $R_{2\pi/m}[A]=A$.}
\end{theorem}
\begin{remark}
  The \textit{initial} support condition (the second condition in \eqref{ass_st_cor})
means that the perturbations should be localized near the Kelvin set $A^{m,\beta}$. Indeed, due to the property $r'>1$, if $\beta>0$ is sufficiently small, then $A^{m,\beta}\subset\subset B_{r'}.$  
Thus, the   condition  
 holds once we simply assume  
that for some small $\mu=\mu(r')>0$, $${A_0\subset \{x\in\mathbb{R}^2\,:\, \mbox{dist}(x, A^{m,\beta})<\mu\}}.	 $$
 (\textit{i.e.} for any $x \in A_0$, there exists a point $x' \in A^{m,\bt}$ such that $|x-x'|<\mu$.)
  The  important point is that  $r'>1$ (so $\mu>0$) is a constant depending only on $m\ge2$ and in particular does not need to become smaller as we vary $\varep$ and $T$. In that sense, the  assumption  is not too much restrictive. However, we note that the condition  seems inevitable for stability as long as we use a variational idea saying that each Kelvin wave is a local maximizer of the kinetic energy in some class. {A counterexample for the case $m=2$ is given in \cite{Tang} (case of the ellipse), and it can be modified to be a counterexample for any $m\ge3$.\\}
\end{remark}
 
When proving the above long time stability, we use three main ingredients. The first one is
the following \textit{conditional orbital stability}, which  {appeared already in Wan's paper \cite[Section 5, Theorem 7]{Wan86} in 1986} with only a very rough sketch of the proof.
{The key observation is that Kelvin waves become non-degenerate local energy maximums when perturbations are assumed to have the same ($m$-fold) symmetry. This idea was already used for Kirchhoff's ellipse  by Tang \cite{Tang} in 1987.}
 We would like to point out that by the time of \cite{Wan86}, even the existence of Kelvin waves have not been rigorously established; the first rigorous existence proof came out only 27 years later in \cite{HMV}. 
We believe that at the time of \cite{Wan86}, a complete proof was impossible {since the proof of stability requires not only existence but also \textit{smoothness} of the boundary of Kelvin waves. In this paper, we provide a detailed and complete proof based on the method suggested in \cite{Wan86}, using rigorous results from \cite{HMV}.}
\begin{proposition}[Conditional Orbital Stability] \label{thm:stable} For each integer $m\geq 2 $, there exist constants $\bar r>1$ and $\bar\beta>0$ satisfying the following property:\\
	
Fix any $0<\bt<\bar{\bt}$. Then, for any $\varepsilon>0$, there exists $\delta>0$ such that if $\omega_0= \mathbf{1}_{A_0}$ where 
\begin{equation*}\label{cond_init}
A_0\mbox{ is } m\mbox{-fold symmetric}, \quad A_0\subset B_{\bar r},\quad\mbox{and}\quad \|\omega_0-\omega^{m,\beta}\|_{L^1(\mathbb{R}^2)}\leq \delta,
\end{equation*} 
then the solution $\omega(t)= \mathbf{1}_{A(t)}$ of \eqref{eq:Euler} with initial data $\omg_0$ is stable up to rotations as long as it is contained in $B_{\bar r}$; more precisely, if for some $0<T\leq\infty$ we have 
\begin{equation}\label{cond_evol}
  \bigcup_{t\in[0,T)} A(t) \subset B_{\bar r} ,
\end{equation} 
then there exists a function $\Theta(\cdot):[0,T)\to  \mathbb{R}$ satisfying
\begin{equation}\label{property_orb}
\sup_{t\in[0,T)}\|\omega(t)-\omega_{\Theta(t)}^{m,\beta}\|_{L^1(\mathbb{R}^2)}\leq \varepsilon. 
\end{equation}
\end{proposition}

{The above result is \textit{conditional} in the sense that the perturbed solution needs to be assumed to stay in  {a given ball (see \eqref{cond_evol}) during the evolution,} and \textit{orbital} as it only gives that the perturbation is $L^1$ close only to \textit{some unspecified} rotation $\Tht(t)$ of the Kelvin wave.\footnote{
These  {two} restrictions have been lifted in Theorem \ref{cor:finite_time}, at the cost of restricting the solution to a finite (but arbitrarily large) time interval.}
 {Even assuming that the second condition can be removed, the first support condition \eqref{cond_evol} on the evolution is fatal in the sense that it seems to be guaranteed only for time of at most order 1.} \\

Let us now focus on how the latter issue is handled: 
orbital stability is very typical when one obtains stability by applying variational idea (\textit{e.g.} see \cite{Tang, Wan86, BNL13, AC2019, Burtonb2021, Choi2020, CQZZ-stab}). Intuitively, it is natural to expect that the perturbed solution is close not only to \textit{some unknown} $\Theta(t)$-rotation of the Kelvin wave but also to the actual  rotating Kelvin wave solution \eqref{act_v_state}. Therefore, to arrive at Theorem \ref{cor:finite_time} from Proposition \ref{thm:stable}, we need an intermediate result which shows that  if   $\Delta t\ll 1$, then
$$\Delta\Theta(t)\quad\mbox{is close to}\quad  \Omega^{m,\beta}\Delta t,\quad 
$$ that is, the perturbed solution \textit{almost} rotates with the angular speed $\Omega^{m,\beta}$ of the Kelvin wave.  {This is the second ingredient of the proof of Theorem \ref{cor:finite_time}.} \\
  
To state the result, we denote $
\mathbb{T}_m = \bbR/(\frac{2\pi}{m}\bbZ)$ by the torus of length $2\pi/m$, which we identify with the interval $[-\frac{\pi}{m},\frac{\pi}{m}).$ Since the Kelvin wave is $m$-fold symmetric, it is natural to assume that the rotation angle $\Omg^{m,\bt}t$ in \eqref{act_v_state} belongs to $\mathbb{T}_m$. Similarly, we shall view the function $\Tht$ as a map 
$\Theta(\cdot):[0,T)\to \mathbb{T}_m$, and denote the projection  $\mathcal{T}_m:\mathbb{R}\to\mathbb{T}_m$ by
$$
\mathcal{T}_m[\alp]=\tilde{\alp}\in\mathbb{T}_m \quad\mbox{if}\quad \alp-\tilde\alp=2k\pi/m\quad\mbox{for some integer}\quad k.
$$ 
\begin{proposition}[Estimate on the required rotation]\label{thm:refined}  
 {(continued from Proposition \ref{thm:stable})} For $m\geq2$, there {exist constants} $C_0, c_0, \beta_0>0$ such that if $\beta\in(0,
\min(\beta_0,\bar \beta)]$ and 
 if $\varepsilon\in(0, c_0\beta^2]$, then the function $\Theta(\cdot_t)$ 
 in Proposition \ref{thm:stable} satisfies
\begin{equation}\label{est_refined}
			|\mathcal{T}_m[\Theta(t'  ) - (\Theta(t) + \Omega^{m,\beta}(t'-t))   ] | \leq C_0\cdot \varepsilon^{1/2}\quad
			  \mbox{whenever}\quad   t,t'\in[0,T)\quad \mbox{satisfy}\quad |t-t'|\leq c_0\beta.
			 	\end{equation}
\end{proposition}
}
 
 {For the third and last ingredient, we  employ the fact that trajectories of the Kelvin wave solution are closed curves, and they stay close to the rotating Kelvin set
$R_{\Omega^{m,\beta}t}[A^{m,\beta}]$ if they are close at the initial time. Then, thanks to conditional stability (Propositions \ref{thm:stable}, \ref{thm:refined}), we are able to prove the long time stability (Theorem \ref{cor:finite_time}). \\}
 
Now, as an application of Theorem \ref{cor:finite_time}, we obtain perimeter growth for finite but arbitrary long time for certain perturbations of the Kelvin wave.  
\begin{theorem}[Instability] \label{thm:instability} For each integer $m\geq 2$, there exists $C'>0$ such that if $\bt>0$ is sufficiently small,
	 then for any $M, \delta>0$, there exists an $m$-fold symmetric data $\omg_0 = \mathbf{1}_{A_0}$ with $C^\infty$ smooth boundary   $\partial A_0$  satisfying \begin{equation*}
	 	\begin{split}
	 		\|\omega_0-\omega^{m,\beta}\|_{L^1(\mathbb{R}^2)}\leq \delta, \qquad \mathrm{perim}({A_0})\leq 20
	 	\end{split}
	 \end{equation*} such that the corresponding solution 	$\omega(t)= \mathbf{1}_{A(t)}$ satisfies \begin{equation*}
		\begin{split}
	 \sup_{t\in[0,C'M]}\mathrm{perim}({A(t)})\geq  M.
		\end{split}
	\end{equation*} 
\end{theorem} 
{We remark that the \textit{improved} stability statement (Theorem \ref{cor:finite_time}), rather than a conditional orbital stability result (Proposition \ref{thm:stable}), is essential in the proof of instability. While we have proved a similar instability result for perturbations of the Lamb dipole in \cite{CJ_lamb} (see also \cite{CJ_Hill} for the case of the Hill's spherical vortex), in this previous work the filament is separating from the center of the perturbed dipole linearly in time, which makes the proof much easier.
}

\subsection{Previous works on stability of $V$--states}

 {The Kelvin waves, and more generally uniformly rotating patch solutions, commonly referred to as ``$V$--states'', to Euler have been intensively studied in the past decades: existence and rigidity of $V$--states (\cite{Burbea1982,HMV,CLW,HHH,HMW,HZTM,HHMV,GS2019,HM16,Tu83,CCG,CCG-duke,EJSVP2,GPSY,Jae1,Jae2,EJS}), linear and nonlinear stability (\cite{Love,Wan86,Tang,CWW,SV09,WP85, CL_radial}), instability (\cite{GuoHS,CoTi,CJ}), numerical computations (\cite{Carrillo2000,Kamm,HZTM,DeZa,ZHR,Drit}). \\}

Let us briefly review the existing results on the stability of $V$--states. The basic idea is to characterize a given $V$--state as the \textit{unique} extremizer of a conserved quantity in an appropriate admissible class. Arnol'd {\cite{A66}} suggested to use the kinetic energy, which is natural since steady solutions are characterized by critical points of the energy. {We also refer to the recent work (\cite{GS21}) for discussions.}
 Some serious work was necessary to apply this idea to the patch case, as it is not a smooth solution to Euler. For the case of the disc $\omg=\mathbf{1}_{B}$, this was achieved by Wan--Pulvirenti (\cite{WP85}) and Tang (\cite{Tang}): the circular patch is actually the unique energy maximizer under a \textit{mass constraint}, which gives nonlinear stability for perturbations in $L^1$. A different approach is to observe that the circular patch is the unique impulse minimizer under a mass constraint, which again yields nonlinear stability in \cite{SV09} (also see \cite{CL_radial}). Indeed, the energy and impulse are two different coercive conservation laws for the two-dimensional vorticity equation. It turns out that, in the case of the ellipse, one can show that upon fixing \textit{both} the mass and impulse, each ellipse for the aspect ratio between 1 and 3 is the unique local maximizer of the energy (\cite{Tang}). The threshold $3$ is sharp, as suggested by previous linear analysis (\cite{Love}) and nonlinear instability for larger aspect ratios (\cite{GuoHS}). \\
 
Under the additional constraint of $m$-fold symmetry, Kelvin waves close to the disc (\textit{i.e.} $0<\beta\ll1$) for each $m\ge3$ can be characterized by the unique local maximum of energy, as stated in Wan \cite{Wan86}. 
 Even though the stability requires $m$-fold symmetry of perturbations, this can be used to prove filamentation simply by taking symmetric perturbations (proof of instability requires stability).  It is interesting that when $m=2$ (\textit{i.e.} Kirchhoff's ellipses), the stability  was obtained not only for  small $\beta>0$ but also up to the aspect ratio 3. It is mainly due to the fact that the stream {function} for each ellipse is explicitly known (\textit{e.g.} see \cite{Lamb}) so that the computation in spectral analysis  in \cite{Tang}  is exact while the   representation \eqref{defn_v_sta} of Kelvin waves for $m\geq3$ works only for small $\beta$. \\
 
  Lastly, we note that when proving stability of steady  solutions of the Euler equations,
  \textit{monotonicity} of the profiles is frequently assumed   
  (\textit{e.g.} see \cite{MaP85,  LinZeng11, BesMas,  Zill17,
Wei_cpam, IonJia_cmp,
 BeCZV19,   CL_radial}) because {this} property gives coercivity in {a} certain sense.
In the Appendix, we demonstrate that the same \textit{spectral} approach can give nonlinear stability for patches supported on an \textit{annulus} by imposing  $m$-fold symmetry for large $m$ to perturbations. 
It is interesting to study stability of such  an annular  patch since it is  a \textit{non-monotone}, radial steady solution.

\subsection*{Organization of the paper} The rest of the paper is organized as follows. In \S \ref{sec:prelim}, we collect a few basic facts about the two-dimensional Euler equations and derive the asymptotic rotation speed of the Kelvin $m$-waves. Then, Proposition \ref{thm:stable}  
is proved in \S \ref{sec:stab}.  Lastly in \S \ref{sec:fila}, we prove
Proposition \ref{thm:refined}, Theorem \ref{cor:finite_time}, and Theorem \ref{thm:instability}. Sections 3 and 4 begin with an overview of the proof. 
In the Appendix, we discuss stability of annular  patches.
\section{Preliminaries}\label{sec:prelim}

\subsection{Two-dimensional incompressible Euler}\label{subsec:Euler}

For the two-dimensional Euler equations, the stream function is defined by \begin{equation*}\label{eq:str}
	\begin{split}
		G[\omg] (x) = (-\lap)^{-1}\omg (x) := \frac{1}{2\pi} \int \ln \frac{1}{|x-x'|} \omg (x') dx' . 
	\end{split}
\end{equation*} When $\omg$ is bounded and compactly supported in $\bbR^2$, then we have that $G[\omg] \in C^{1,\alp}_{loc}(\bbR^2)$ with any $\alp<1$. The energy functional is defined by \begin{equation*}\label{eq:energy}
\begin{split}
	E[\omg] = \frac12 \brk{\omg, G\omg}= \frac{1}{4\pi} \int_{\bbR^2} \int_{\bbR^2} \omg(x) \omg(x') \ln \frac{1}{|x-x'|}  dx dx' .
\end{split}
\end{equation*} For bounded solutions to \eqref{eq:Euler} decaying sufficiently fast at infinity, it is not difficult to check that $E$ is a conserved quantity in time. We just remark that, strictly speaking, $E$ is not the kinetic energy of the fluid unless $\omg$ is of mean zero in $\bbR^{2}$: in this case, we have \begin{equation*}
\begin{split}
	E[\omg] = \frac12 \int_{\bbR^2} |\nb G\omg |^{2} dx \ge0. 
\end{split}
\end{equation*} In general, $E[\omg]$ is not positive and this quantity is sometime referred to as the \textit{pseudo-energy} in the literature.

\subsection{Kelvin waves}\label{subsec:Kelvin}
{Let $r_0>0$ and } denote $B$ by the open ball with radius $r_0$  centered at the origin. Then, it is not difficult to check that in polar coordinates, the corresponding stream function is given by \begin{equation*}\label{eq:disc-stream}
	\begin{split}
		G [\mathbf{1}_{B}] (r,\tht) = \begin{cases}
			\frac14 (r_0^2 - r_0^2 \ln r_0^2 - r^2), & 0 \le r \le r_0, \\
			-\frac12 r_0^2 \ln r,  & r > r_0. 
		\end{cases}
	\end{split}
\end{equation*} Let us now revisit the computation of Kelvin, and assume that there exists a uniformly rotating patch $\omg^{m,\bt}$ with boundary $r_{0} + g(\tht)$ with $g \simeq \bt \cos(m\tht)$. We shall {give a sketch of the derivation of} the rotation speed $\Omg^{m,\bt}$ in the limit $\bt\to0$. {A rigorous argument is given in \cite{HMV}.} For some $C>0$, once we define the \textit{relative stream function} by 
\begin{equation}\label{eq:psi-mbt}
	\begin{split}
		\psi^{m,\bt} = G\omg^{m,\bt} + \frac12 \Omg^{m,\bt} r^2 + C  ,
	\end{split}
\end{equation} so that for all $\tht \in [0,2\pi]$, we have \begin{equation*}
\begin{split}
	\psi^{m,\bt}(r_{0}+ g(\tht),\tht) \equiv 0.
\end{split}
\end{equation*}  {Introducing for convenience $\zt := G(\omg^{m,\bt}-\mathbf{1}_{B})$}, by differentiating the above relation in $\tht$, we obtain \begin{equation}\label{eq:stream-relation}
\begin{split}
	0 = (-\frac12 + \Omg^{m,\bt})r_{0} \rd_\tht g^{m,\bt}(\tht) + \rd_{r}\zt   \rd_\tht g^{m,\bt}(\tht) + \rd_{\tht}\zt . 
\end{split}
\end{equation} Note that  
\begin{equation*}
\begin{split}
	\zt (r,\tht) = \frac{1}{2\pi} \int_{S^1} \int_{r_{0} }^{r_{0} + g(\tht')}     \ln \frac{1}{|re^{i\tht} - r' e^{i\tht'} | } r' dr' d\tht'.
\end{split}
\end{equation*} Assuming $\int \omg^{m,\bt} = \int \mathbf{1}_{B}$, we may expand the above in the case $r > r_{0} - |g(\tht')|$ as follows:  using $	g(\tht) = \bt\cos(m\tht) + o(\bt)$ and that the small term is orthogonal to $1, \cos(m\tht)$, 
\begin{equation*}
	\begin{split}
		\zt(r,\tht) &=  \frac{1}{2\pi} \int_{S^1} \int_{r_{0} }^{r_{0} + g(\tht')}  \Re \sum_{n\ge1} \frac{1}{n} \left( \frac{r'}{r} \right)^{n} e^{in(\tht'-\tht)}  r' dr' d\tht' = \frac{1}{2m} \frac{r_{0}^{m+1}}{r^{m}}\bt \cos(m\tht) + o(\bt). 
	\end{split}
\end{equation*}  Similarly, one may compute that \begin{equation*}
\begin{split}
	\rd_{\tht}\zt (r,\tht)= -\frac{r_0^{m+1}}{2 r^{m}} \bt \sin(m\tht) + o(\bt), \qquad 	\rd_{r}\zt (r,\tht) = -\frac{r_{0}^{m+1}}{2r^{m+1}}  \bt \cos(m\tht) + o(\bt). 
\end{split}
\end{equation*} Since we know that $\zt \in C^{1,\alp}$ for any $\alp<1$, these formulas can be justified up to the boundary of the rotating patch. Applying these to \eqref{eq:stream-relation}, we obtain that 
\begin{equation*}
	\begin{split}
		\Omg^{m,\bt} = \frac12 - \frac{1}{2m} +  {o(\bt)} . 
	\end{split}
\end{equation*} {The remainder term seems to be of order $\bt^2$ (see \cite{Burbea1982}), but we shall not need this fact in what follows.}

\section{Stability of Kelvin waves}\label{sec:stab}

\subsection{Outline of the proof of Proposition \ref{thm:stable}}\label{subsec:stab-outline}

This section is devoted to the proof of the stability result. To state the key proposition, let us first define the following class of perturbations, fixing some $m\ge2$ and $\bt>0$. The value of $\bt$ will be assumed to be sufficiently small whenever it becomes necessary, but in a way depending only on $m$. {(Inspecting the proof, one can see that $\bt \lesssim m^{-2}$ is sufficient.)} We take \begin{equation*}
	\begin{split}
		\calM(\omg^{m,\bt}) := \left\{  \omg = \mathbf{1}_{A} \, : \, \int r^{m}\sin(m\tht)\omg(x) dx = 0, \int r^2( \omg(x) - \omg^{m,\bt} ) = 0,\int ( \omg(x) - \omg^{m,\bt} ) = 0 \right\}.
	\end{split}
\end{equation*} Then, we say $\omg \in \calM_m(\omg^{m,\bt})$ if $\omg \in \calM(\omg^{m,\bt}) $ and furthermore $\omg$ is $m$-fold symmetric. Next, given an open set $D$, we define the class 
\begin{equation*}
	\begin{split}
		 \calN_{\varep,D}(\omg^{m,\bt}) := \left\{  \omg =\mathbf{1}_{A}\, : \, A\subset D, \nrm{\omg-\omg^{m,\bt}}_{L^{1}}< \varep\right\}. 
	\end{split}
\end{equation*} We are now ready to state the key technical result of this section, which shows that within the class $\calN_{\varep,D}\cap \calM_{m}(\omg^{m,\bt})$ for  {$D = B_{\bar{r}r_0}$ with some $\bar{r}>1$}, $\omg^{m,\bt}$ is the \textit{strict} maximizer of the energy.  {As before, we shall fix $r_0 = 1$ for simplicity.}
\begin{proposition}\label{prop:key}
	For any $m\ge 2$, there exist $\bt_{0}>0$ and $\bar{r}>1$ depending on $m$ such that the following statement holds. For the Kelvin $m$-wave with parameter $0<\bt<\bt_0$, there exist $C, \varepsilon >0$ depending on $m, \bt$ such that \begin{equation}\label{eq:key}
		\begin{split}
			E[\omg^{m,\bt}] - E[\omg] \ge C \nrm{\omg^{m,\bt} - \omg }_{L^{1}}^{2}
		\end{split}
	\end{equation} for any $\omg \in \calN_{\varep,B_{\bar{r}}}\cap \calM_{m}(\omg^{m,\bt})$. 
\end{proposition}

In \S \ref{subsec:stab}, we show how our main stability (Proposition \ref{thm:stable}) follows from the above proposition, which is rather straightforward. Then, the remainder of this section is devoted to establishing Proposition \ref{prop:key}. The structure of the argument is parallel to that for the ellipse stability by Tang \cite{Tang},  {which corresponds to the case $m=2$}, and mainly consists of two steps: (i) reduction to a graph-type perturbation and (ii) energy comparison for graph type perturbations. To be more precise, given $\omg$ satisfying the assumptions of Proposition \ref{prop:key}, we shall find $\tilde{\omg}$ such that \begin{equation}\label{eq:tomg-1}
	\begin{split}
		E[\omg^{m,\bt}] - E[\tilde{\omg}] \ge C \nrm{\omg^{m,\bt} - \tilde{\omg} }_{L^{1}}^{2}
	\end{split}
\end{equation} and \begin{equation}\label{eq:tomg-2}
	\begin{split}
		-E[\omg] + E[\tilde{\omg}] \ge C \nrm{\omg - \tilde{\omg} }_{L^{1}}^{2}
	\end{split}
\end{equation} holds. Combining the above two inequalities, we obtain \eqref{eq:key}. We shall identify such a $\tilde{\omg}$ and prove \eqref{eq:tomg-2} in \S \ref{subsec:redu}. Then, \eqref{eq:tomg-1} is proved in \S \ref{subsec:spec}: this part is the heart of the matter and requires a spectral analysis of the linearized operator coming from the Green's function for the Laplacian. 

In what follows, we shall use a simple change of variables $(r,\tht) \to (\xi,\eta) $ near $\{ r = r_0\}$, so that $\eta(r,\tht)=\tht$ and $\xi(r,\tht) = r - g^{m,\bt}(\tht)$. Then, $\{ \xi = r_{0} \}$ corresponds to $\partial A^{m,\bt}$. The Jacobian $J = J^{m,\bt}$ from $x = (x_1,x_2)$ to $(\xi,\eta)$ is $\xi + g^{m,\bt}(\tht) = \xi + \bt \cos(m\tht) + o(\bt)$. 

\subsection{Proof of stability} \label{subsec:stab}
In this section, we show how Proposition \ref{thm:stable} follows readily from Proposition \ref{prop:key}. This procedure is straightforward, the key point being that the energy difference is controlled by the $L^1$ difference of vorticities. First, we prove an intermediate result, namely nonlinear $L^1$ stability under the natural constraints on the initial vorticity. 

\begin{lemma}\label{lem:stab}
	Fix some $m\ge2$ and assume that $0<\bt$ is sufficiently small. Then, for any sufficiently small $\varep>0$, there exists $\dlt>0$ such that for $\tilde{\omg}_{0} \in \calM_m \cap \calN_{\dlt,B_{\bar{r}}}(\omg^{m,\bt}),$ we have \begin{equation}\label{eq:stab}
	\begin{split}
		\nrm{ \tilde{\omg}(t, \cdot) - \omg^{m,\bt}_{t'} }_{L^{1}} < \varep,{\quad\forall t\geq0,} 
	\end{split}
\end{equation} for some $t'=t'(t)$, provided that $\supp( \tilde{\omg}(s,\cdot) ) \subset B_{\bar{r}}$ for all $s \in [0,t]$. 
\end{lemma}
Note that for two vorticities $\omg$ and $\tilde{\omg}$ which are compactly supported in $\bbR^2$, we have \begin{equation}\label{eq:energy-bound}
	\begin{split}
		|E[\omg] - E[\tilde{\omg}]| = \frac12\left| \brk{\omg-\tilde{\omg} , G[\tilde{\omg}] } - \brk{ \omg, G[\tilde{\omg}-\omg] }  \right| \le C \nrm{\omg-\tilde{\omg}}_{L^1}
	\end{split}
\end{equation} where $C>0$ depends on the radius of the support (see \cite[Lemma 5.1]{Tang}). 

\begin{proof}[Proof of Lemma \ref{lem:stab} assuming Proposition \ref{prop:key}]
	This is nothing but \cite[Lemma 5.3]{Tang}, although we provide a simplified argument. Let us suppose that $\tilde{\omg}_{0}$ verifies the assumptions in the above. Furthermore, it will be convenient to consider the Euler equations in a rotating frame in which $\omg^{m,\bt}$ becomes a steady state, and denote $\tilde{\omg}(t,\cdot)$ to be the solution defined under this frame. Note that the solution $\tilde{\omg}(t,\cdot)$ belongs to the class $\calM_{m}$, possibly except for the condition \begin{equation}\label{eq:moment}
		\begin{split}
			\int r^{m}\sin(m\tht) \tilde{\omg}(t,\cdot) dx= 0.
		\end{split}
	\end{equation}

	\medskip
	
	\noindent \underline{Proof of \eqref{eq:stab} under the assumption of \eqref{eq:moment} and $\tilde{\omg}(t) \in \calN_{\varep,B_{\bar{r}}}$.} For the moment, assume that \eqref{eq:moment} holds at some time $t$ and $\tilde{\omg}(t) \in \calN_{\varep,B_{\bar{r}}}$. Then, from Proposition \ref{prop:key} and \eqref{eq:energy-bound}, we derive \begin{equation*}
	\begin{split}
	 \frac{1}{C}\nrm{\omg^{m,\bt} - \tilde{\omg}_0}_{L^{1}}\ge   E[\omg^{m,\bt}]- E[\tilde{\omg}_0]   =	E[\omg^{m,\bt}]- E[\tilde{\omg}(t)] \ge C \nrm{\omg^{m,\bt} - \tilde{\omg}(t)}_{L^{1}}^{2}
	\end{split}
\end{equation*} and hence \begin{equation*}
\begin{split}
	\nrm{\omg^{m,\bt} - \tilde{\omg}(t)}_{L^{1}} \le C\dlt^{\frac12} < \frac{\varep}{4},
\end{split}
\end{equation*} where the last inequality follows simply by taking $\dlt>0$ small in a way depending on $\varep$.

	\medskip
	
	\noindent \underline{Removing the additional assumptions.} We observe that the quantity $\nrm{\omg^{m,\bt} - \tilde{\omg}(t)}_{L^{1}}$ is \textit{Lipschitz continuous} in time, which follows from the fact that the boundary of the support of $\omg^{m,\bt}$ is smooth and that the velocity of $\tilde{\omg}(t)$ is uniformly bounded in time. Therefore, from the continuity,  one can take some small $T>0$ such that on $[0,T]$, we have \begin{equation}\label{eq:bootstrap}
		\begin{split}
			\nrm{\omg^{m,\bt} - \tilde{\omg}(t)}_{L^{1}} < \frac{\varep}{2}. 
		\end{split}
	\end{equation} Since the condition $\supp( \tilde{\omg}(s,\cdot) ) \subset B_{\bar{r}}$  is given in the statement of the lemma, we obtain on $[0,T]$ that  $\tilde{\omg}(t) \in \calN_{\varep,B_{\bar{r}}}$. Then, at $t = T$, it is not difficult to see that by rotating $\tilde{\omg}(T)$ with some small angle $\tau$ (taking $\varep$ smaller if necessary), we can arrange that \begin{equation*}
\begin{split}
				\int r^{m}\sin(m\tht) \tilde{\omg}_\tau (T,\cdot) dx = 0.
\end{split}
\end{equation*} See the last part of the proof of Lemma \ref{lem:reduction-ImFT} for the details of this argument. Recalling the argument in the above, this shows that we can actually upgrade the estimate \eqref{eq:bootstrap} to \begin{equation*}
\begin{split}
	\nrm{\omg^{m,\bt}_{-\tau} - \tilde{\omg}(T)}_{L^1} =\nrm{\omg^{m,\bt} - \tilde{\omg}_\tau(T)}_{L^1} < \frac{\varep}{4}.
\end{split}
\end{equation*} Since we may choose $T$ depending only on $\omg^{m,\bt}$ and $\varep$ (using \textit{Lipschitz continuity} in time of the quantity $\nrm{\omg^{m,\bt} - \tilde{\omg}(t)}_{L^{1}}$), we may inductively  obtain bounds of the $L^1$ difference on time intervals $[T,2T]$, $[2T, 3T]$, and so on. 
\end{proof}
\begin{proof}[Proof of Proposition \ref{thm:stable} from Lemma \ref{lem:stab}]
We first note that Lemma \ref{lem:stab} works  for general $r_0>0$ by rescaling. Then the idea for proof of Proposition \ref{thm:stable}    is to simply ``adjust'' both the initial data $\omega_0$ and 
  the Kelvin wave $\omega^{m,\beta}$   in a way that we are reduced to the setup of Lemma \ref{lem:stab}. 
Given $\omg_{0}$ satisfying the assumptions of Proposition \ref{thm:stable}, we  
 may find $\lmb, \bt', \tau$ verifying $$|\lmb-1|, |\bt'-\bt|, |\tau| \ll 1$$ (from the inverse function theorem) such that 
the rotated initial data $(\omega_0)_\tau(x)=\omega_0(R_{-\tau}x)$ and  
 the rescaled, $\beta$-reparametrized wave $\omega^{m,\beta',\lambda}(x):=\omega^{m,\beta'}(\lambda x)$
 satisfy 
$$\int r^{m}\sin(m\tht)(\omg_0)_\tau(x) dx =0,	$$ and 
 \begin{equation*}
\begin{split}
	\int {\omg}_0( x) dx = \int \omg^{m,\bt',\lambda}(x) dx, \qquad 	\int |x|^2{\omg}_0(x) dx = \int |x|^2 \omg^{m,\bt',\lambda} (x) dx,
\end{split}
\end{equation*} respectively,  by taking $\dlt>0$ smaller if necessary depending on $m, \bt$. 
 We observe that {if we set  $\tilde{\omg}_0:=(\omega_0)_\tau$, we have}
\begin{equation*}
\begin{split}
  \nrm{ \tilde\omg_0 - \omg^{m,\bt',\lambda} }_{L^1}& = \nrm{ \omg_0 - \omg^{m,\bt',\lambda}_{-\tau} }_{L^1}\\ &\leq
  \nrm{ \omg_0 - \omg^{m,\bt} }_{L^1}
+\nrm{ \omg^{m,\bt}-\omg^{m,\bt'} }_{L^1} 
+\nrm{ \omg^{m,\bt'}-\omg^{m,\bt',\lambda} }_{L^1}  
  +\nrm{\omg^{m,\bt',\lambda}- \omg^{m,\bt',\lambda}_{-\tau} }_{L^1},
\end{split}
\end{equation*} {and the right hand side can be made} arbitrarily small by assuming $\delta>0$ small again. Thus we get 
$\tilde{\omg}_{0} \in \calM_m \cap \calN_{\dlt,B_{\bar{r}}}(\omg^{m,\bt',\lambda})$ 
 so that we  can apply   Lemma \ref{lem:stab} (for general $r_0>0$) to {$\tilde{\omg}_{0}$} with the {Kelvin} wave $\omega^{m,\beta',\lambda}$, which gives \begin{equation*}
\begin{split}
	\nrm{ \tilde{\omg} (t) - \omg_{t'}^{m,\bt',\lambda} }_{L^1} < \frac{\varep}{3} 
\end{split}
\end{equation*}
  for some angle $t' = t'(t)$. 
 (For this, we may take $\bar{r}>0$ in the statement of Proposition \ref{thm:stable} slightly smaller than the original  $\bar r>0$ given in Proposition \ref{prop:key}.)   
    Then by choosing appropriate  angle  $t_1 = t_1(t')$,
   we have \begin{equation*}
\begin{split}
		\nrm{ {\omg} (t) - \omg^{m,\bt}_{t_1-\tau} }_{L^1} &= \nrm{ \tilde{\omg} (t) - \omg^{m,\bt}_{t_1} } _{L^1}\\ &\le 	\nrm{ \tilde {\omg} (t) -\omega^{m,\beta',\lambda}_{t'} }_{L^1} + 	\nrm{ \omega^{m,\beta',\lambda}_{t'} - \omega^{m,\beta'}_{t'}  }_{L^1} + 	\nrm{ \omega^{m,\beta'}_{t'}  - \omg^{m,\bt}_{t_1} }_{L^1}\\
		&\le \frac \varepsilon 3 + 	\nrm{ \omega^{m,\beta',\lambda} - \omega^{m,\beta'}  }_{L^1} + 	\nrm{ \omega^{m,\beta'}  - \omg^{m,\bt} }_{L^1}.  
\end{split}
\end{equation*} The last two  terms on the right hand side can be taken to be less than $\varep/3$ by choosing $\delta$ small. This finishes the proof. 
\end{proof}

\subsection{Reduction to graph perturbations}\label{subsec:redu}
We set $\calE$ to be a sufficiently small neighborhood of $\psi^{m,\bt}$ in the $C^1$--topology, where $\psi^{m,\bt}$ is the relative stream function of $\omg^{m,\bt}$ defined in \eqref{eq:psi-mbt}. Before we proceed, observe that the set $\{ \psi^{m,\bt} > 0 \}$ consists of two components, with the inner one describing the set $A^{m,\bt}$. Since the gradient of $\psi^{m,\bt}$ is non-degenerate on $\partial A^{m,\bt}$ for $\bt$ small, we have that for any $\psi \in \calE$, the inner component of $\{ \psi > 0 \}$ is an open set close to $A^{m,\bt}$. {Note that, as we take $\bt\to 0$, the relative stream functions $\psi^{m,\bt}$ converge in $C^{1,\alp}$ to the limit $\psi^{m,0}$: \begin{equation*}
	\begin{split}
		\lim_{\bt\to0} \psi^{m,\bt} = G [ \mathbf{1}_B] + \frac{1}{2}\Omg^{m,0} r^2 + C^{m,0} = - \frac12 \ln r + \frac12 \left( \frac12 - \frac1{2m} \right) (r^2 - 1). 
	\end{split}
\end{equation*} The function described on the right hand side is strictly positive on $0<r<1$, negative on $1<r<r^*$, and again positive on $r^*<r$, for a constant $r^*>1$ depending only on $m$. We conclude that, once we pick any $1 <\bar{r} < r^*$, then there exists $\bar{\bt} > 0$ such that for any $0< \bt < \bar{\bt}$, the relative stream function $\psi^{m,\bt}$ is strictly negative in the region $B_{\bar{r}} \backslash A^{m,\bt}$. When $m = 2, 3, 4$, we have that $r^* \simeq 1.87, 1.46, 1.32$, respectively. }

\begin{lemma}[The reduction lemma]\label{lem:reduction-ImFT}
	Given $\omg_{1} = \mathbf{1}_{A_1} \in \calN_{\varep,B_{\bar{r}}} \cap \calM_{m}(\omg^{m,\bt})$, there exists a $C^1$--smooth $\tilde{\psi}$ close to $\psi^{m,\bt}$ defined in \eqref{eq:psi-mbt} such that if we denote the {inner} 
	component of $\{ \tilde{\psi} \ge 0 \}$ by $\tilde{A}$, then \begin{equation*}
		\begin{split}
			\tilde{\omg} := \mathbf{1}_{\tilde{A}} \in \calM_{m}(\omg^{m,\bt}){\quad\mbox{and}\quad \brk{\tilde{\omega}-\omega_1, G\omega_1-\tilde \psi}=0.}
		\end{split}
	\end{equation*}
\end{lemma}

\begin{proof} Given {$\psi\in \calE$,} we define 
	\begin{equation*}
		\begin{split}
			\psi_{\mu} := \psi + \frac12 \mu_1 r^2 +  \mu_2. 
		\end{split}
	\end{equation*} Given $\mu = (\mu_1, \mu_2)$ which is close to $\mathbf{0} = (0,0)$, we may set $A_{\psi,\mu}$ to be the inner component of $\{ \psi_{\mu} \ge 0 \}$ and $\omg_{\psi,\mu} = \mathbf{1}_{A_{\psi,\mu}}$. Define for a small neighborhood $\calO\subset \bbR^2$ of the origin, \begin{equation*}
	\begin{split}
		F = (F_1,F_2) : \calE \times \calO \to \bbR^2
	\end{split}
\end{equation*} by \begin{equation*}
\begin{split}
	F_{1}(\psi,\mu) := \int r^{2} (\omg_{\psi,\mu} - \omg^{m,\bt}) dx , \qquad 	F_{2}(\psi,\mu) := \int (\omg_{\psi,\mu} - \omg^{m,\bt}) dx . 
\end{split}
\end{equation*} Assuming that $\calO$ is smaller if necessary, we have that $\partial A_{\psi,\mu}$ is described by a graph $h = h_{\psi,\mu}: S^{1}\to\bbR$ with \begin{equation*}
\begin{split}
	\psi_{\mu} ( r_{0} + h(\eta) , \eta ) = 0.
\end{split}
\end{equation*} In particular, we have that $h_{\psi^{m,\bt},0} \equiv 0$. Based on this, we compute \begin{equation*}
	\begin{split}
		\rd_{h} \psi_{\mu}|_{ (\psi,\mu)  = (\psi^{m,\bt},0) } & = \rd_{\xi} \psi^{m,\bt}|_{\xi = r_{0}} = \rd_{r}\psi^{m,\bt} (r_{0} + g^{m,\bt}(\tht),\tht)   \\
		& = -\frac{1}{2m}(r_{0} + \bt\cos(m\tht)) + \rd_{r}G ( \omg^{m,\bt} - \mathbf{1}_{B} ) (r_{0}+g^{m,\bt}(\tht),\tht) \\
		& = -\frac{1}{2m}(r_{0} + \bt\cos(m\tht)) - \frac12 \bt\cos(m\tht)  + o(\bt). 
	\end{split}
\end{equation*}  
Then, from \begin{equation*}
	\begin{split}
		\rd_{\mu_{j}} h = \frac{ \rd_{\mu_{j}} \psi_{\mu}  }{ \rd_{h}\psi_{\mu} }
	\end{split}
\end{equation*} we obtain that  \begin{equation*}\label{eq:h-der}
\begin{split}
	\rd_{\mu_{1}} h = \frac{r^2}{2\rd_h\psi},   \quad \rd_{\mu_{2}} h = \frac{1}{\rd_h\psi}.
\end{split}
\end{equation*}
This allows us to compute $\rd_\mu F$ at $(\psi^{m,\bt},\bf0)$. For convenience we introduce the notation $f = o^\perp(A)$ to mean that the function $f$ satisfies $|f| \ll A$ and $f$ is orthogonal in $L^2(S^1)$ with $1$ and $\cos(m\tht)$.  To begin with,  \begin{equation*}
	\begin{split}
		\rd_{\mu_{1}} F_{1} & = \int_{0}^{2\pi} (r^{2} J)|_{\xi=r_{0}}  \rd_{\mu_{1}} h d\eta = \int_{0}^{2\pi} \frac{(r_{0}+ \bt\cos(m\tht))^4 (r_{0} + \bt\cos(m\tht) + o^\perp(\bt))}{-\frac{1}{m}r_0 -( 1 + \frac{1}{m}) \bt\cos(m\tht) + o^\perp(\bt)}  d\tht \\
		& = -m \int_{0}^{2\pi} (r_0^4 + 4r_0^3 \bt\cos(m\tht) + 6 r_0^2\bt^2 \cos^2(m\tht) + o(\bt^2) )(1 + \frac{\bt}{r_0} \cos(m\tht)+o^\perp(\bt) ) \\
		& \qquad\qquad  \times (1 - (1+m)\frac{\bt}{r_{0}} \cos(m\tht) + (m+1)^2 \frac{\bt^2}{r_0^2} \cos^2(m\tht) + o(\bt^2) ) d\tht . 
	\end{split}
\end{equation*} Then, this gives \begin{equation*}
\begin{split}
		\rd_{\mu_{1}} F_{1} &= -mr_0^4 \int_0^{2\pi} 1 + \left( 4 - 5(1+m) + 6 + (1+m)^2 \right) \frac{\bt^2}{r_0^2} \cos^2(m\tht) d\tht + o(\bt^2) \\
	& = -mr_0^4 \left( 2\pi + (m(m-3) + 6 )\frac{\bt^2}{r_0^2} \pi \right) + o(\bt^2). 
\end{split}
\end{equation*} Next, one can similarly compute that \begin{equation*}
\begin{split}
		\rd_{\mu_{2}} F_{1} & = \int_{0}^{2\pi} (r^{2}J)|_{\xi=r_{0}}  \rd_{\mu_{2}} h d\eta = -mr_0^2 \left( 2\pi + (1-m+m^2)\frac{\bt^2}{r_0^2} \pi \right) + o(\bt^2),
\end{split}
\end{equation*}  \begin{equation*}
\begin{split}
	\rd_{\mu_{1}} F_{2} & = \int_{0}^{2\pi} J|_{\xi=r_{0}}  \rd_{\mu_{1}} h d\eta = -mr_0^2 \left( 2\pi + (1-m+m^2)\frac{\bt^2}{r_0^2} \pi \right) + o(\bt^2),
\end{split}
\end{equation*}  and  \begin{equation*}
\begin{split}
	\rd_{\mu_{2}} F_{2} & = \int_{0}^{2\pi} J|_{\xi=r_{0}}  \rd_{\mu_{2}} h d\eta = -m \left( 2\pi + (1+m+m^2)\frac{\bt^2}{r_0^2} \pi \right) + o(\bt^2).
\end{split}
\end{equation*} Therefore, we conclude that \begin{equation*}
\begin{split}
	\det (\nb F) =  (2\pi m r_0^2)^2 \left(  \frac{5}{2} \frac{\bt^2}{r_0^2} + o(\bt^2) \right). 
\end{split}
\end{equation*} In particular, there exists some $\bt_0>0$ so that for $\bt \in (0,\bt_0)$, $\det (\nb F)>0$. Fixing such a $\bt$ and applying the inverse function theorem to the map $F$ at $(\psi,\mu) = (\psi^{m,\bt}, \bf0)$, we obtain existence of a unique 
{$$ \tilde{\psi}:=(G\omega_1)_\mu=G\omega_1+ \frac12 \mu_1 r^2 +  \mu_2.  $$}
close to $\psi^{m,\bt}$ such that the corresponding vorticity $\tilde{\omg}$ satisfies {$F(G\omega_1,\mu) = 0$,} namely \begin{equation}\label{eq:first-two}
\begin{split}
	\int r^2 (\tilde{\omg} - \omg^{m,\bt}) dx = 0, \qquad \int (\tilde{\omg} - \omg^{m,\bt}) dx = 0. 
\end{split}
\end{equation} It is clear that $\tilde{\psi}$ (and therefore $\tilde{\omg}$) is $m$-fold rotationally symmetric. For $\tilde{\omg}$ to belong to the class $\calM_{m}$, it still remains to verify the condition \begin{equation*}
\begin{split}
	\int r^{m} \sin(m\tht) \tilde{\omg} dx = 0. 
\end{split}
\end{equation*} This is done by rotating $\tilde{\psi}$ around the origin; that is, define $\tilde{\psi}_{\tau}(r,\tht) :=  \tilde{\psi}(r, \tht+\tau)$ in polar coordinates and denote the corresponding vorticity (defined as the characteristic set of the inner component of $\{ \tilde{\psi}^\tau\ge0\}$) by $\tilde{\omg}^{\tau}$. Observe that \begin{equation*}
\begin{split}
		\int r^{m} \sin(m\tht) {\omg^{m,\bt}} dx = 0
\end{split}
\end{equation*} and since $\tilde{\psi}$ is close to $\psi^{m,\bt}$ in the $C^1$ topology, we have \begin{equation*}
\begin{split}
	\left| 	\int r^{m} \sin(m\tht) \tilde{\omg} dx  \right| \ll 1, \qquad \left| \frac{d}{d\tau} \left( \int r^{m} \sin(m\tht) \tilde{\omg}_{\tau} dx  - \int r^{m} \sin(m\tht) \omg^{m,\bt}_{\tau} dx   \right) \right| \ll 1.
\end{split}
\end{equation*} Here, $\ll 1$ means that the constant can be arbitrarily small by taking $\varep\to 0$ where $\varep$ is from $\calN_{\varep,B_{\bar{r}}}$. Since \begin{equation*}
\begin{split}
	\frac{d}{d\tau}\int r^{m} \sin(m\tht) \omg^{m,\bt}_{\tau} dx    = \frac{d}{d\tau} \int r^{m} \sin(m\tht-m\tau) \omg^{m,\bt} dx    = \int r^{m} m\cos(m\tht-m\tau) \omg^{m,\bt} dx   
\end{split}
\end{equation*} is strictly positive at $\tau = 0$, we can find some $\tau$ satisfying $|\tau|\ll1$ such that \begin{equation*}
\begin{split}
		\int r^{m} \sin(m\tht) \tilde{\omg}_{\tau} dx = 0. 
\end{split}
\end{equation*} Observe that rotating around the origin does not alter \eqref{eq:first-two}. The proof is complete. 
\end{proof}

\begin{lemma}\label{lem:energy-compare1} Given $\omg_{1}$ satisfying the assumptions of Lemma \ref{lem:reduction-ImFT}, let $\tilde{\omg}$ to be the associated graph-type vorticity
{from Lemma \ref{lem:reduction-ImFT}.} Then, we have 
	\begin{equation*}\label{eq:energy-diff-pert}
		\begin{split}
			E[\tilde{\omg}] - E[\omg_{1}] \ge C\nrm{\tilde{\omg} - \omg_{1}}_{L^{1}}^{2} . 
		\end{split}
	\end{equation*}
\end{lemma}
\begin{proof}
	Using the formula for the energy difference, we proceed as follows: \begin{equation*}
		\begin{split}
				E[\tilde{\omg}] - E[\omg_{1}] &= \brk{ \tilde{\omg} - \omg_{1}, G\omg_{1} } + \frac12 \brk{   \tilde{\omg} - \omg_{1}, G( \tilde{\omg} - \omg_{1}) }   \ge  \brk{ \tilde{\omg} - \omg_{1}, G\omg_{1} } \\
				&{=}  \brk{ \tilde{\omg} - \omg_{1}, \tilde{\psi} } = \int_{ \tilde{A} \backslash A_{1} } \tilde{\psi} -  \int_{ A_{1} \backslash \tilde{A} } \tilde{\psi} \ge  \int_{ \tilde{A} \backslash A_{1} } \tilde{\psi} .
		\end{split}
	\end{equation*} It is important to note that the assumption $\supp(\omg_1)\subset B_{\bar{r}}$ is used to guarantee that $\tilde{\psi}\le 0$ on $A_1\backslash\tilde{A}$. 
From a uniform lower bound for $\rd_r \tilde{\psi}$ near $\partial\tilde{A}$, it is not difficult to show that the last expression has a lower bound of the form $ C\nrm{\tilde{\omg} - \omg_{1}}_{L^{1}}^{2} $, since $\tilde{\psi}=0$ on $\partial\tilde{A}$ and $|\tilde{A} \backslash A_{1}| \gtrsim \nrm{\tilde{\omg}-\omg_{1}}_{L^1}$. 
\end{proof}

\subsection{Spectral analysis}\label{subsec:spec}

In this section, we shall consider \textit{graph-type} perturbations of $\omg^{m,\bt}$. For this purpose, it will be convenient to work on the coordinate system $(\xi,\eta)$ adapted to $\omg^{m,\bt}$, after fixing some $(m,\bt)$ with $m\ge2$ and $\bt>0$ sufficiently small in a way depending on $m$. Furthermore, $S^{1}$ will denote the set $\{ (\xi, \eta) : \xi = r_{0}, 0 \le \eta <2\pi \}$ in $\bbR^{2}$, unless otherwise specified. Let $h \in C^{1}(S^{1})$ be a function with sufficiently small $C^{1}$--norm in the $\eta$ variable. In this section, let us use the notation \begin{equation*}
	\begin{split}
		\omg_{h} := \mathbf{1}_{A_{h}}, \qquad A_{h} := \left\{ (\xi,\eta):  \xi \le r_{0} + h(\eta) \right\}.
	\end{split}
\end{equation*} For $h$ sufficiently small in $C^{1}$, the closed set $A_{h}$ is well-defined and close to the set $A^{m,\bt}$. In this notation, note that we have $\omg^{m,\bt} = \omg_{\bf0}$ where $\bf0$ is the zero function on $S^1$. 

Now observe that $\omg_{h}$ is $m$-fold symmetric in $\bbR^{2}$ if and only if $h$ is $m$-fold symmetric in the sense that $h(\eta) = h(\eta+ \frac{2\pi}{m})$ for any $\eta\in S^{1}$. For such a function $h$, we have the following simple result. 
\begin{lemma}\label{lem:m-symm}
	Let $h$ be $m$-fold symmetric on $S^{1}$. Then for any integer $0< n<m$, we have that \begin{equation*}
		\begin{split}
			\int_{S^1} e^{in\eta} h(\eta)d\eta=0. 
		\end{split}
	\end{equation*}
\end{lemma}
\begin{proof}
	With the change of variables $\eta\to \eta+\frac{2\pi}{m}$, \begin{equation*}
		\begin{split}
				\int_{S^1} e^{in\eta} h(\eta)d\eta=	\int_{S^1} e^{in\eta+ 2\pi i \frac{n}{m}} h(\eta+\frac{2\pi}{m})d\eta=	\int_{S^1} e^{in\eta+ 2\pi i \frac{n}{m}} h(\eta)d\eta.
		\end{split}
	\end{equation*} This gives \begin{equation*}
	\begin{split}
		(1-e^{2\pi i\frac{n}{m}})\int_{S^1} e^{in\eta} h(\eta)d\eta=0. 
	\end{split}
\end{equation*}When $0< n<m$, we have that $e^{2\pi i\frac{n}{m}}\ne 1$ and we are done.
\end{proof}

The following result from Tang \cite{Tang} gives the expansion of the energy for graph-type perturbations. Since it applies to general rotating solutions of Euler, the Lemma is directly applicable in our case. 
\begin{lemma}[{{\cite[Lemma 4.1]{Tang}}}] \label{lem:energy-expansion}
	Let $\omg^* = \mathbf{1}_{A^*}$ be a rotating patch solution where $\partial A^*$ is described by a smooth graph $g^*$.
{Let $\psi^*$	be the relative stream \eqref{eq:psi-mbt} of the rotating patch $\omg^*$.}
	 Furthermore, let $(\xi,\eta)$ be a coordinate system defined near $\partial A^*$ satisfying $\eta=\tht$ and $\{ \xi=r_{0} \} = \partial A^*$, and $J_{0}$ is the Jacobian of $x\mapsto (\xi,\eta)$ restricted to $S^{1} := \{ \xi=r_{0} \}$.  Consider $C^{1}$ graph-type perturbations of $\omg^*$, namely $\omg_{h}$ satisfying $\nrm{h}_{C^{1}(S^1)} \ll 1$. Furthermore, assume that we have \begin{equation}\label{eq:second-order}
		\begin{split}
			\int (\omg_h - \omg^*) \, dx = 0, \quad 	\int x(\omg_h - \omg^*) \, dx = 0, \quad 	\int |x|^2(\omg_h - \omg^*) \, dx = 0. 
		\end{split}
	\end{equation} Then, we have that for $q(\eta) := J_{0}(\eta) h(\eta)$, \begin{equation}\label{eq:energy-expansion}
		\begin{split}
			E[\omg_h] - E[\omg^*] = \frac12 \brk{q, \calL q} + o(\nrm{h}_{L^2}^2), 
		\end{split}
	\end{equation} where  \begin{equation*}\label{eq:L-def}
\begin{split}
	\calL q := I_{0} q + \int_{S^1} K(\eta,\eta') q(\eta') d\eta' 
\end{split}
\end{equation*} with \begin{equation*}\label{eq:I-K-def}
\begin{split}
	I_{0} := \frac{ \rd_{\xi} {\psi^*}|_{\xi=r_0} }{ J_{0}} , \qquad K(\eta,\eta') := \frac{1}{2\pi} \ln \frac{1}{| x(r_0,\eta) - x(r_0,\eta') |} .
\end{split}
\end{equation*} Here, $\brk{,}$ denotes the $L^2$ inner product on $S^{1}$. 
\end{lemma}
 
Given the above key lemma, we are in a position to conclude the main result of this section. 

\begin{proposition}\label{prop:key-graph}
	For $\omg_{h}  \in \calN_{\varep,B_{\bar{r}}}(\omg^{m,\bt}) \cap \calM_{m}(\omg^{m,\bt})$ with $h \in C^{1}(S^{1})$, we have \begin{equation*}\label{eq:energy-graph}
		\begin{split}
			E[\omg^{m,\bt}] - E[\omg_{h}] \ge C \nrm{\omg^{m,\bt} - \omg_{h}}_{L^{1}}^{2}
		\end{split}
	\end{equation*} for some $C>0$. 
\end{proposition}
\begin{proof}
	Note that $\nrm{h}_{L^2}$ and $\nrm{q}_{L^2}$ are equivalent up to constants. 
{We apply Lemma \ref{lem:energy-expansion} with $\omega^*=\omega^{m,\bt}$ and} proceed in several steps. 
	
	\medskip

	\noindent \underline{Step 1: Cancellation conditions.} To begin with, we note that the first condition from \eqref{eq:second-order} implies \begin{equation}\label{eq:imp1}
		\begin{split}
			0 = \int_{0}^{2\pi} \int_{r_0}^{r_0+h } J d\xi d\eta ,
		\end{split}
	\end{equation} which gives after expanding $J(\xi,\eta) = J_0(\eta) + O(|\xi-r_0|)$ and integrating in $\xi$, \begin{equation}\label{eq:q0}
	\begin{split}
		\int_{S^{1}} q (\eta) d\eta = O(\nrm{h}_{L^2}^2). 
	\end{split}
\end{equation} Similarly, the last condition from \eqref{eq:second-order} gives \begin{equation*}
\begin{split}
	0 =  \int_{0}^{2\pi} \int_{r_0}^{r_0+h } (\xi + g^{m,\bt}(\eta))^{2} J(\xi,\eta) d\xi d\eta .
\end{split}
\end{equation*} Writing $\xi = r_{0} + (\xi - r_{0})$, applying \eqref{eq:imp1} and expanding $J$ as above, we obtain that \begin{equation*}\label{eq:qm2}
\begin{split}
	2r_{0}\int_{S^{1} } g^{m,\bt} q d\eta = O(\nrm{h}_{L^2}^2) + o(\bt\nrm{h}_{L^2}).
\end{split}
\end{equation*} That is, \begin{equation}\label{eq:qm-cos}
\begin{split}
	\int_{S^{1} } \cos(m\eta) q(\eta) d\eta = O(\bt^{-1}\nrm{h}_{L^2}^2) + o(\nrm{h}_{L^2}).
\end{split}
\end{equation} Next, from the condition \begin{equation*}
\begin{split}
	\int r^{m} \sin(m\tht) \omg_{h} dx = 0,
\end{split}
\end{equation*} we obtain that \begin{equation*}
\begin{split}
	0 =  \int_{0}^{2\pi} \int_{r_0}^{r_0+h } (r_{0} + g^{m,\bt}(\eta) + (\xi-r_0))^{m} \sin(m\eta) J(\xi,\eta) d\xi d\eta .
\end{split}
\end{equation*} Then, it follows \begin{equation}\label{eq:qm-sin}
\begin{split}
	\int_{S^{1}} \sin(m\eta) q(\eta) d\eta = O(\bt\nrm{h}_{L^2}) + O(\nrm{h}_{L^2}^2). 
\end{split}
\end{equation} Lastly, applying Lemma \ref{lem:m-symm} to $q$ (note that $J_0$ is $m$-fold symmetric and so does $q$) gives \begin{equation}\label{eq:qn}
\begin{split}
	\int_{S^{1}} e^{in\eta} q(\eta) d\eta = 0, \qquad 0<n<m. 
\end{split}
\end{equation} 

\medskip

\noindent \underline{Step 2: Computation for $I_{0}$.} We compute that \begin{equation*}
		\begin{split}
			I_0 = -\frac{1}{r_0+ \bt\cos(m\tht){+o(\beta)}} \left(\frac{r_0}{2m} + (\frac{1}{2} +\frac{1}{2m})\bt\cos(m\tht)+{o(\bt)}\right) = - \left(  \frac{1}{2m} + {(\frac12+\frac1{2m})} \frac{\bt}{r_{0}} \cos(m\tht) + {o(\bt)} \right). 
		\end{split}
	\end{equation*} This gives \begin{equation}\label{eq:I0-est}
	\begin{split}
		\nrm{ I_{0}q + \frac{1}{2m} q }_{L^{2}} \le C\bt\nrm{q}_{L^2}. 
	\end{split}
\end{equation}

\medskip

\noindent \underline{Step 3: Computation for $K$.} We shall replace $K$ with $K^*$ up to an $O(\bt)$ error, which is the convolution operator arising in the disc case. The operator $K^*$ simply corresponds to the case $\bt = 0$. To this end, we first note that using the condition \eqref{eq:q0}, we have that \begin{equation*}
	\begin{split}
		K[q] (\eta) & := \int_{S^{1}} K(\eta,\eta') q(\eta')d\eta'   = - \frac{1}{2\pi} \int_{S^{1}} \ln\left| ({r_{0} + g^{m,\bt}(\eta)}) e^{i\eta} - {(r_0+ g^{m,\bt}(\eta'))}e^{i\eta'} \right| q(\eta') d\eta' \\ 
		& = - \frac{1}{2\pi} \int_{S^{1}} \ln\left| 1 - \frac{r_0+ g^{m,\bt}(\eta')}{r_{0} + g^{m,\bt}(\eta)} e^{i(\eta'-\eta)} \right| q(\eta') d\eta' + O(\nrm{h}_{L^2}^2). 
	\end{split}
\end{equation*} We define \begin{equation*}
\begin{split}
	K^*[q](\eta) := - \frac{1}{2\pi} \int_{S^{1}} \ln\left| 1 -  e^{i(\eta'-\eta)} \right| q(\eta') d\eta' .
\end{split}
\end{equation*} Then, with pointwise bounds \begin{equation*}
\begin{split}
	\left| 1 -  \frac{r_0+ g^{m,\bt}(\eta')}{r_{0} + g^{m,\bt}(\eta)} \right| \le C\bt |\eta'-\eta|, \qquad 	\left| 1 -  e^{i(\eta'-\eta)} \right| \ge c|\eta'-\eta|, 
\end{split}
\end{equation*}  we obtain that \begin{equation*}
\begin{split}
	\left| 	K^*[q] - 	K[q]   \right|(\eta) \le C\bt \nrm{q}_{L^1} \le C\bt\nrm{q}_{L^2}. 
\end{split}
\end{equation*}

\medskip

\noindent \underline{Step 4: Coercivity.} From the previous step and \eqref{eq:I0-est}, we have \begin{equation*}
	\begin{split}
		\brk{\calL q,q} & \le \brk{I_0q, q} + \brk{K^*q, q} + C\bt\nrm{q}_{L^2}^2  \le - \frac{1}{2m} \nrm{q}_{L^2}^2 + \brk{K^*q, q} + C\bt\nrm{q}_{L^2}^2. 
	\end{split}
\end{equation*} We now expand $q$ in Fourier series \begin{equation*}
	\begin{split}
		q = \sum_{n\in\bbZ} q_{n} e^{in\eta}. 
	\end{split}
\end{equation*} Since $q$ is real, we have that $q_{-n} = \overline{q_{n}}$. Now, we recall the exact formula 
{
$$\frac 1 {2n}=-\frac 1 {2\pi}\int_0^{2\pi}\ln\left|1-e^{i\eta'}\right|\,e^{in\eta'}d\eta',\quad n>0$$ so that 
}
\begin{equation*}
\begin{split}
	\brk{K^*q, q} = \alp_0 |q_0|^2 + \sum_{n\in\bbZ\backslash \{ 0\}} \frac{1}{2n} |q_{n}|^{2} = I + II, {\quad \alp_0:=K^*1,}
\end{split}
\end{equation*} where for some $C_0>0$ depending on $m$, we have \begin{equation*}
\begin{split}
	II := \sum_{|n|>m} \frac{1}{2n} |q_{n}|^{2} < \left(\frac{1}{2m} - C_{0} \right) \sum_{|n|>m} |q_n|^{2}. 
\end{split}
\end{equation*} Next, \begin{equation*}
\begin{split}
	I :=  \alp_0 |q_0|^2 +  \sum_{0<|n|\le m} \frac{1}{2n} |q_{n}|^{2} \le C\bt^2\nrm{q}_{L^2}^2, 
\end{split}
\end{equation*} using \eqref{eq:q0}, \eqref{eq:qm-cos}, \eqref{eq:qm-sin} and \eqref{eq:qn}, and taking $\nrm{h}_{L^2}\le \bt^{2}$. Then, using the Plancherel theorem, we continue estimating as follows: \begin{equation*}
\begin{split}
	\brk{\calL q,q} & \le  - \frac{1}{2m} \nrm{q}_{L^2}^2+  \left(\frac{1}{2m} - C_{0} \right) \sum_{|n|>m} |q_n|^{2}  + C\bt\nrm{q}_{L^2}^2 \le - \frac{C_{0}}{2} \nrm{q}_{L^2}^2,
\end{split}
\end{equation*} by taking $\bt>0$ smaller if necessary in a way depending only on $C_{0}$. (Recall that $C_0$ depends only on $m$.) 

\medskip

\noindent \underline{Step 5: Completion of the proof.} From \eqref{eq:energy-expansion} in Lemma \ref{lem:energy-expansion}, we have that \begin{equation*}
	\begin{split}
		E[\omg^{m,\bt}] -E[\omg_h] = -\frac{1}{2} \brk{\calL q, q} + o(\nrm{h}_{L^2}^2) \ge \frac{C_0}{8}\nrm{q}_{L^2}^2. 
	\end{split}
\end{equation*} However, it is clear that \begin{equation*}
\begin{split}
	\nrm{ \omg_h - \omg^{m,\bt} }_{L^{1}} = \int_0^{2\pi} \left| \int_{r_0}^{r_0+ h} J d\xi \right| d\eta \le C\nrm{q}_{L^2}(1 + \nrm{q}_{L^2}). 
\end{split}
\end{equation*} For $\nrm{q}_{L^2}$ small, we conclude that \begin{equation*}
\begin{split}
	E[\omg^{m,\bt}] -E[\omg_h] \ge \nrm{ \omg_h - \omg^{m,\bt} }_{L^{1}}^{2}.
\end{split}
\end{equation*} This finishes the proof. \end{proof}

\section{Refined stability and filamentation}\label{sec:fila}
In this entire section, we fix an integer $m\geq 2$ so that every estimate and constant appeared in this section may depend on the choice of the integer $m$ even though we do not specify the dependency on $m$ for simpler presentation. When considering a  Kelvin wave, we always assume $r_0=1$ so that $A^{m,\beta}=\{r<1+g^{m,\beta}(\theta)\}.$ Let us give an outline of the arguments. 

\subsection{Outline of the proof}
\noindent  \textbf{Refined stability (Proposition \ref{thm:refined})}\\
To prove  the refined estimate \begin{equation}\label{idea_concl}
	|\mathcal{T}_m[\Delta\Theta(t)-\Omega^{m,\beta}\Delta t   ] | \lesssim\varepsilon^{1/2},\quad
	\,\Delta t =\mathcal{O}(\beta),
\end{equation}
we combine  a bootstrap argument with the orbital stability result (Proposition \ref{thm:stable}).
Indeed, we first show that the degree of adaptive rotation $\Theta$ cannot change significantly over a \textit{small} period of time (Lemma \ref{lem:no_high_jump}):
$$
|\mathcal{T}_m[\Delta \Theta(t) ]  | \lesssim  \varepsilon^{1/2}, \quad   \,\Delta t =\mathcal{O}(\varepsilon)\quad \mbox{if}\quad \varepsilon\lesssim \beta^2.
$$ It means that our perturbed solution behaves very similarly to the rotating Kelvin wave at least for a short period of time (of order $\varepsilon$). Since the ``Kelvin set'' $A^{m,\beta}$ rotates exactly under its own flow map, we can show that if we leave the  set $A^{m,\beta}$ in the perturbed flow from the perturbed solution for a short time, then the set lies on a \textit{small} neighborhood of the precisely rotated Kelvin set
$A^{m,\beta}_{\Omega^{m,\beta}t}$:
\begin{equation*} 
	\begin{split}
		\phi(t,   A^{m,\beta} )\subset  
		\{ x\in\mathbb{R}^2\,:\, \dist(x, A^{m,\beta}_{ \Omega^{m,\beta} t})
		\lesssim \beta\cdot\varepsilon^{1/2}\}.
	\end{split}
\end{equation*}   
This detailed information leads to the above conclusion \eqref{idea_concl}.

\medskip
\noindent  \textbf{Unconditional stability up to finite time without any adjusting rotation (Theorem \ref{cor:finite_time})}\\
For any fixed time $T>0$, we show that the rotating Kelvin wave is stable in $L^1$-sense without any adjusting rotation and without a condition on the evolution such as \eqref{cond_evol}. 
To do this, we add the above refined estimate  \eqref{idea_concl} for small time repeatedly to derive a finite time result:
\begin{equation*}
	|\mathcal{T}[\Theta(t )- \Omega t]    | \lesssim \left(\frac{T}{\beta}+1\right) \varepsilon^{1/2}\quad 			  \mbox{for all}\quad   t\in[0,T).
\end{equation*}
However,  it requires 
that  the perturbation should remain for the given time duration in a certain small ball containing the Kelvin wave (see the condition \eqref{cond_evol}). By observing the dynamics of the Kelvin wave and by comparing it with the perturbed one, we derive the \textit{initial} condition \eqref{ass_st_cor} that guarantees the hypothesis during the evolution.

\medskip
\noindent  \textbf{Filamentation (Theorem \ref{thm:instability})}\\
To prove perimeter growth of boundary, we recall that the Kelvin waves are close to the unit disk when the parameter $\beta>0$ is sufficiently small. We also note that the angular velocity of the disk has a  non-trivial derivative in the radial direction outside the disk. As is well known, the further out of the disk, the slower the angular speed. This idea was already used in \cite{CJ} when deriving an example of perimeter growth near the disk. Similarly, we take two points from the boundary of a perturbed patch, and trace their trajectories. From the above finite time stability,
each trajectory remains arbitrary close to the original orbit from the Kelvin wave  for a large desired amount of time.  This process is possible by assuming that the perturbation is small enough in $L^1$. 
When considering any curve lying on the initial boundary connecting these points, the curve is transported    by the perturbed flow so that 
its length increases by the difference  multiplied by time.

\subsection{Notations for Kelvin wave and simple estimates}\label{subsec:fil_pre}
If $\beta>0$ is small enough so that the Kelvin wave $\omega^{m,\beta}$ exists, then we simply denote,\begin{equation*}
	\begin{split}
		\mathbb{T}=\mathbb{T}_m, \qquad \mathcal{T}[\cdot_\alp]=\mathcal{T}_m[\cdot_\alp]:\mathbb{R}\to\mathbb{T}_m, \qquad {\Omega}=\Omega^{m,\beta},\qquad g = g^{m,\beta},
	\end{split}
\end{equation*}  
$$\bar{A}=A^{m,\beta}, \quad \bar{A}_\alpha=\mathrm R_{\alpha}[\bar{A}]\quad\mbox{for}\quad  \alpha\in \mathbb{R},$$  where $\mathrm R_{\alp}$ is the (counter-clockwise) rotation map by the angle $\alp$, and
$$
\bar{\omega}_\alp
:= \mathbf{1}_{\bar A_{\alp}}, \quad  \bar{\omega}=\bar{\omega}_0:
= \mathbf{1}_{\bar A}.
$$
 We also set 
\begin{equation*}
\bar{I}_{\alpha}:=\int_{\mathbb{R}^2}e^{im\theta}\bar{\omega}_{\alpha}(x)dx
=\int_{\bar{A}_\alp}e^{im\theta} dx
\in \mathbb{C},\quad \bar{I}:=\bar{I}_0\in\mathbb{C}.
\end{equation*} Here we use the polar coordinate $(r,\theta)$ for $x\in\mathbb{R}^2$. 
Then
it is easy to check, for each $\alpha\in\mathbb{R}$, 
\begin{equation}\label{id:bar_i_alp}
 \bar{I}_{\alpha}= \bar{I}e^{im\alpha}=\bar{I}_{\mathcal{T}[\alp]}.
\end{equation}
We collect some properties of Kelvin waves. 
\begin{lemma} \label{lem:v_st} 
There exist constants $c_i >0$ for  $i=1,\dots, 5$ such that
\begin{equation}\label{est_bar_i}
c_1\beta\leq |\bar{I} |\leq c_2\beta,
\end{equation}
\begin{equation}\label{est_bar_i_alp}
c_3\cdot|\mathcal{T}[\alp]|\cdot \beta\leq |\bar{I}-\bar{I}_\alp |,
 \quad \forall \alp\in
 \mathbb{R},
\end{equation}
\begin{equation}\label{est_v_st_der}
\sup_{\theta\in \mathbb{T}}| g'(\theta)|\leq c_4, 
\end{equation}
\begin{equation}\label{est_v_st_der2}
|\bar A\setminus \bar A_\alp|\leq c_5\mathcal{T}[\alp],\quad\forall\alp\in\mathbb{R},
\end{equation} 
hold for any sufficiently small $\beta>0$.

\end{lemma}

 \begin{proof}

From the representation \eqref{defn_v_sta} of $\bar\omega$, we 
get \eqref{est_bar_i}. Then, by  \eqref{id:bar_i_alp}, we have
$$
 |\bar{I}-\bar{I}_\alp |=|\bar I||1-e^{im\alp}|,
$$ which gives  \eqref{est_bar_i_alp}. 
Lastly, \eqref{est_v_st_der}, \eqref{est_v_st_der2} follow from 
 $g^{m,\beta}\to g^{m,0}\equiv 0$ in $C^1$ as $\beta\to0$ (\textit{e.g.} see   \cite{HMV}).
 \end{proof}

 For   $\eta\geq 0$ and for $\alp\in\mathbb{R}$,  we denote the  $\eta$-neighborhood of $\bar A_{\alp} $ by
 \begin{equation}\label{defn_eta}
\bar A^{\eta}_{\alp}:=\{
x\in\mathbb{R}^2\,:\, 
dist(x,\bar A_{\alp})<\eta
\},\quad \bar A^{\eta}:=\bar A^{\eta}_{0}.
 \end{equation}
 When $\beta>0$ is sufficiently small, then we observe that
$$\bar A^\eta\subset \{(r,\theta): r\leq 1+g(\theta)+C\eta\}$$
for some  
$C>0$
thanks to the estimate 
\eqref{est_v_st_der} in Lemma \ref{lem:v_st}. It implies
\begin{lemma}\label{lem:C2}
There exists some $C_1>0$ such that
\begin{equation}\label{est_bar_eta}
|\bar A_\alp^\eta \setminus \bar A_\alp|=|\bar A^\eta \setminus \bar A|\leq C_1\eta,\quad \alp\in\mathbb{R}
\end{equation}
 holds for any sufficiently small $\beta>0$.
\end{lemma}
We also denote $\Omega^*=\Omega^*(m)=\frac{m-1}{2m}>0$, and observe
$$\Omega=\Omega^{m,\beta}\to \Omega^* \quad\mbox{as} \quad \beta\to 0.$$ From now on, we always assume $\beta>0$ sufficiently small   to have Lemmas
\ref{lem:v_st}, \ref{lem:C2} and to satisfy $\frac{1}{2}\Omega^*\leq\Omega\leq 2\Omega^*.$

For any  given $(t_0,x)\in\mathbb{R}_{\geq0}\times \mathbb{R}^2$, we denote the trajectory map $\bar\phi(\cdot_t,(t_0,x))$ from the  Kelvin wave  solution $\bar\omega(t)= \mathbf{1}_{\bar A_{\Omega t}}$  with $\bar u:=K *\bar \omega$    by solving \begin{equation}\label{defn_barpsi}
	\left\{
	\begin{aligned}
		&\frac{d}{dt}\bar\phi(t,(t_0,x))=\bar u(t,\bar \phi(t,(t_0,x))),\\&\bar \phi(t_0,(t_0,x))=x.
	\end{aligned}
	\right.
\end{equation}
We remark that $\bar u$ is Lipschitz in space-time from regularity of $\partial \bar A$, and    $$ \bar\phi(t,(t_0,\bar A_{\Omega t_0}))= \bar A_{\Omega t}, \quad \forall t, \forall t_0\geq 0.$$  
Lastly, we take any constant $\hat C>0$ such that  any function $f\in L^1\cap L^\infty(\mathbb{R}^2)$ satisfies
  \begin{equation}\label{defn_hat_C}
 \|\frac{1}{|x|}* f\|_{L^\infty(\mathbb{R}^2)}\leq \hat C \left(\|f\|_{L^1}\|f\|_{L^\infty}\right)^{1/2}
   \end{equation}  (\textit{e.g.} see Lemma 2.1 of \cite{isg}).

 \subsection{Only small jumps in $\Theta(t)$}

When considering an  initial data $\omega_0= \mathbf{1}_{A_0}$ for some  $m-$fold symmetric open set $A_0\subset \mathbb{R}^2$, we set 
\begin{equation*}
I(t)=\int_{\mathbb{R}^2}e^{im\theta}\omega(t,x)dx\in \mathbb{C},
\end{equation*}
where $\omega(t)= \mathbf{1}_{A(t)}$ is the corresponding solution. 
As in \eqref{defn_barpsi}, for given $(t_0,x)\in\mathbb{R}_{\geq0}\times \mathbb{R}^2$, the trajectory map $\phi(\cdot_t,(t_0,x))$ for the solution $\omega(t)$ is defined  by \begin{equation}\label{trajec_pertu}
\left\{
\begin{aligned}
	&\frac{d}{dt}\phi(t,(t_0,x))=u(t,\phi(t,(t_0,x)),\\
	&\phi(t_0,(t_0,x))=x.
\end{aligned}
\right.
\end{equation}

We observe that the adjusting function $\Theta$ in Proposition \ref{thm:stable} satisfying \eqref{property_orb} may not be continuous. Even, it does not have to be uniquely determined.  We first prove that the function $\Theta$ is allowed to have \textit{at most} small jumps of order $\sqrt{\varepsilon}$ (up to $2\pi/m$-additions).
\begin{lemma}\label{lem:no_high_jump} 
There exist constants $\tilde{\beta}>0$, $ \tilde{K}>0$, and $ \tilde{C}>0$   
satisfying  the following statement: \\

Let  
   $\beta\in(0,\tilde{\beta}]$ and
$\bar{\omega}=\omega^{m,\beta}$ be the Kelvin wave with $r_0=1$. 
 If 
a $m-$fold symmetric solution $\omega(t)= \mathbf{1}_{A(t)}$ with a function $\Theta:[0,T)\rightarrow\mathbb{R}$ for some $0<T\leq\infty$ satisfies 
\begin{equation}\label{ass_st_lem}
\sup_{t\in[0,T)}\|\omega(t)-\bar{\omega}_{\Theta(t)}\|_{L^1(\mathbb{R}^2)}\leq \varepsilon
\end{equation} for some $\varepsilon\in(0,\beta^2]$,
	then 
the function $\Theta$  satisfies
 \begin{equation}\label{eq:no_high_jump}
	\begin{split}
		|\mathcal{T}[\Theta(t) -\Theta(t')]  | \leq \tilde{K}\cdot \varepsilon^{1/2} 
	\end{split}
\end{equation} whenever $t,t'\in[0,T)$ satisfies $|t-t'|\leq \tilde{C}\cdot \varepsilon$.
\end{lemma}
\begin{proof}
Let $\tilde{\beta}\in(0,\bar \beta]$ be sufficiently small to satisfy 
all the estimates in   \S \ref{subsec:fil_pre}, where $\bar \beta>0$ comes from Proposition \ref{thm:stable}, and consider $\beta\in(0,\tilde{\beta}]$.
 For a simple presentation, we denote
 $$\Theta=\Theta(t), \quad \Theta'=\Theta(t'), \quad\omega=\omega(t),\quad \omega'=\omega(t').$$
\begin{enumerate}

\item 

We first remark that
$$ \bar I_{\mathcal{T}[\Theta-\Theta']} = \bar I_{\Theta-\Theta'}\quad\mbox{and}\quad
|\bar I - \bar I_{\Theta-\Theta'}|=
|\bar I_{\Theta} - \bar I_{\Theta'}|
$$ so that
 the conclusion \eqref{eq:no_high_jump} follows once we prove
\begin{equation*}
|\bar I_{\Theta} - \bar I_{\Theta'}|\leq  c_3\beta \tilde{K}\cdot \varepsilon^{1/2}
\end{equation*} thanks to \eqref{est_bar_i_alp}  in Lemma \ref{lem:v_st} ($c_3>0$ is the constant from the lemma).
\item We begin the estimate
\begin{equation}\label{est_i_dif}\begin{split}
|\bar I_{\Theta} - \bar I_{\Theta'}|\leq 
|\bar I_{\Theta} -I(t)|+|I(t)-I(t')|+|I(t') - \bar I_{\Theta'}|.
\end{split}\end{equation}
By the stability assumption \eqref{ass_st_lem}, we estimate the first term by
\begin{equation*}\begin{split}
 |\bar I_{\Theta} -I(t)|\leq \int_{\mathbb{R}^2}|\bar{\omega}_{\Theta}-\omega|dx\leq \varepsilon.
\end{split}\end{equation*} Similarly, 
$ |\bar I_{\Theta'} -I(t')|\leq  \varepsilon.$
For the term in the middle, we estimate
\begin{equation*}\begin{split}
 |I(t)-I(t')| \leq \|\omega-\omega'\|_{L^1}=|A(t')\triangle A(t)|=2|A(t')\setminus A(t)|,
\end{split}\end{equation*} where $|\cdot|$ is the Lebesgue measure in $\mathbb{R}^2$. Then, for the particle trajectory map $\phi$ \eqref{trajec_pertu} from the solution $\omega(t)$, we note
\begin{equation*}\begin{split}
A(t')=\phi(t',(t,A(t))),\quad  A(t)=(\bar A_{\Theta} \cap A(t))\cup( A(t)\setminus\bar A_{\Theta} ),
\end{split}\end{equation*}
$$
 |\bar A_{\Theta} \setminus A(t)|\leq \|\bar\omega_\Theta-\omega(t)\|_{L^1}\leq \varepsilon,
$$
 and
$$
|\phi(t',(t,(A(t)\setminus\bar A_{\Theta})))|=|A(t)\setminus\bar A_{\Theta}| \leq \|\bar\omega_\Theta-\omega(t)\|_{L^1}\leq \varepsilon.
$$
Thus, we can estimate 
\begin{equation}\label{est_a'_a}\begin{split}
|A(t')\setminus A(t)|&\leq |A(t')\setminus\bar A_{\Theta} |+|(\bar A_{\Theta} \setminus A(t))|\\
&\leq 
 |\phi(t',(t,(\bar A_{\Theta} \cap A(t))))\setminus \bar A_{\Theta} |+ |\phi(t',(t, ( A(t)\setminus \bar A_{\Theta})) |
+\varepsilon\\
&\leq 
 |\phi(t',(t,\bar A_{\Theta})  )\setminus\bar A_{\Theta} |+ 2\varepsilon.
\end{split}\end{equation} 
\item
We recall the flow speed is    uniformly bounded  for all time:
$$\sup_{t\geq 0}\|u(t)\|_{L^\infty}\leq  C\sup_{t\geq 0}\|\omega(t)\|^{1/2}_{L^1}\|\omega(t)\|_{L^\infty}^{1/2}\leq C_2<\infty$$ for some  $C_2>0$.
Now we take  $\tilde C:=(2C_2C_1)^{-1}$ and $\tilde{K}:=7/c_3.$ 
 
As a consequence of the previous step, we get
$$\phi(t',(t,\bar A_{\Theta})) \subset \bar A_{\Theta}^{C_2|t-t'|},
$$ which gives, from \eqref{est_a'_a} and from \eqref{est_bar_eta}, 
\begin{equation*} \begin{split}
|A(t')\setminus A(t)| 
&\leq 
 | \bar A_{\Theta}^{C_2 |t-t'|}\setminus\bar A_{\Theta} |+ 2\varepsilon\leq C_1 C_2|t'-t|+2\varepsilon.
\end{split}\end{equation*} 
Thus, from \eqref{est_i_dif}, for $|t-t'|\leq \tilde C\varepsilon$,
\begin{equation*}\begin{split}
|\bar I_{\Theta} - \bar I_{\Theta'}| &\leq 
 |I(t)-I(t')|+2\varepsilon\leq 2|A(t')\setminus A(t)| +2\varepsilon\\&
 \leq 2 C_2C_1\tilde C \varepsilon +6\varepsilon=7\varepsilon\leq 7\sqrt\varepsilon\beta\leq  c_3\beta \tilde K \varepsilon^{1/2}.
\end{split}\end{equation*} 
 
\end{enumerate}
 This finishes the proof.  
 \end{proof}
 \subsection{Proof of Proposition \ref{thm:refined}}
 
Now we will prove \textit{refined stability} (Proposition \ref{thm:refined}) using \textit{orbital stability} (Proposition \ref{thm:stable}) and Lemma \ref{lem:no_high_jump}.
\begin{proof}[Proof of Proposition \ref{thm:refined}]
We prove the result by a bootstrap argument.

\begin{enumerate}
\item  We set  $\beta_0=\min\{\tilde{\beta},\bar\beta\}$ where $\tilde{\beta},\bar\beta>0$ are the constants from 
Lemma \ref{lem:no_high_jump} and Proposition \ref{thm:stable}, respectively.
We take $c_0\in(0,1]$, which will be chosen sufficiently small during the proof (see \eqref{defn_c_0}).
 For 
$\beta\in(0,\beta_0]$, 
we consider a $m$-fold symmetric solution $\omega(t)=\mathbf{1}_{A(t)}$ with a function $\Theta:[0,T)\to\mathbb{R}$ 
\begin{equation}\label{ass_st_proof}
\sup_{t\in[0,T)}\|\omega(t)-\bar{\omega}_{\Theta(t)}\|_{L^1(\mathbb{R}^2)}\leq \varepsilon
\end{equation} for some $0<T\leq\infty$ and 
for some  $\varepsilon\in(0,c_0\beta^2]$.  

\item Fix any $t_0\in[0,T)$. We will find some constants ${c_0}, C_0>0$, which are independent of the choice of $t_0$, satisfying the following property:

	\noindent \textbf{Goal}. For all $t \in [t_0,t_0+{c_0\beta} ]\cap [0,T)$,
   \begin{equation} \label{goal_boot} 
		\begin{split}
		|\mathcal{T}[\Theta(t ) - (\Theta(t_0) + \Omega(t-t_0))  ] | \leq C_0\cdot \varepsilon^{1/2} .
		\end{split}
	\end{equation} 
For the rest of the proof, every time variable is assumed to be on $[0,T)$.

	
	
	\item First, we prove the following claim:\\
		\indent  \textbf{Initial claim}. 
There exists	 a constant  $\eta=\eta(\varepsilon)>0$  such that,  for any $t\in[t_0,t_0+\eta]$, we have
	\begin{equation} \label{init_boot} 
		\begin{split}
		|\mathcal{T}[\Theta(t ) - (\Theta(t_0) + \Omega(t-t_0)) ]  | \leq \frac 1 2 C_0\cdot \varepsilon^{1/2}.
		\end{split}
	\end{equation} 
	
	This estimate \eqref{init_boot}  directly follows from Lemma \ref{lem:no_high_jump}. Indeed, the lemma implies that
		\begin{equation*}  
		\begin{split}
	|\mathcal{T}[\Theta(t ) - (\Theta(t_0) + \Omega(t-t_0))  ] | &\leq
	 			|\mathcal{T}[\Theta(t ) -\Theta(t_0)]| + |\mathcal{T}[\Omega(t-t_0)]   | \\ &\leq \tilde{K}\varepsilon^{1/2}+ 2\Omega^*|t-t_0|\quad\mbox{whenever}\quad |t-t_0|\leq \tilde C \varepsilon,
		\end{split}
	\end{equation*} 
	  where $\tilde C, \tilde K$ come from Lemma \ref{lem:no_high_jump}.
	We just take any constants $C_0>0$ large  and $\eta=\eta(\varepsilon)>0$ small to satisfy 
	 $$
	 \tilde K\leq \frac 1 4 C_0,\quad \eta\leq \tilde C \epsilon,\quad \mbox{and}\quad 2\Omega^* \eta\leq   \frac{1}{4}C_0 \varepsilon^{1/2}, 
	 $$
	which gives \eqref{init_boot}.

\item

From now on, we may assume that \eqref{goal_boot}  is valid for $t \in [t_0,t^*]$ with some $t^*>t_0$, \textit{i.e.}
\begin{equation}\label{assu_boot_cla}
	|\mathcal{T}[\Theta(t ) - (\Theta(t_0) + \Omega(t-t_0))]   | \leq C_0\cdot \varepsilon^{1/2},\quad \forall t\in[t_0,t^*]. 
\end{equation}
 The existence of such a moment $t^*>t_0$ is guaranteed by \textbf{Initial claim} \eqref{init_boot}.  We shall prove the following \textit{bootstrap} claim: \ \\

\indent  \textbf{Bootstrap claim}.   There exists a small  constant ${c_0}>0$   such that if
\eqref{assu_boot_cla} holds for some 
 $t^* \leq  t_0 + {c_0\beta} 
$, 
 then    we have for all $t\in[t_0,t^*]$,
\begin{equation} \label{hyp_boot}
		\begin{split}
		|\mathcal{T}[\Theta(t ) - (\Theta(t_0) + \Omega(t-t_0)) ]  | \leq \frac 1 2 C_0\cdot \varepsilon^{1/2} .		\end{split}
	\end{equation}   We note that the coefficient of $\sqrt{\varepsilon}$ in \eqref{assu_boot_cla} is $C_0$ while that  in \eqref{hyp_boot} is $(1/2)C_0$.

	 \item
Before proving 
\eqref{hyp_boot}, we will perform a refined estimate for the  trajectory map  $\phi$ \eqref{trajec_pertu} from the solution $\omega(t)$.  First, we set the  constant  $\gamma_0>0$ by
\begin{equation}\label{defn_gam}
		\gamma_0:= \left(\frac 1 2c_3  \frac 1 2 C_0\right)/(2C_1)>0,
\end{equation} 
where $C_1>0$ is the constant from    \eqref{est_bar_eta} of Lemma \ref{lem:C2}. 
 Then we claim, for any $t\in[t_0,t^*]\subset[t_0,t_0+{c_0\beta}]$,
	  	 \begin{equation} \label{pill}
		\begin{split}
\phi(t,(t_0,\bar A_{\Theta(t_0)}))\subset 
\bar A_{\Theta(t_0)+\Omega(t-t_0)}^{{\gamma_0\cdot\beta}\cdot\varepsilon^{1/2}}.
		\end{split}
	\end{equation}    	 Here,   the superscript to a set $\bar A_\alp$ represents the neighborhood of the set (see \eqref{defn_eta}).
	
	\item To prove \eqref{pill}, we fix
any $x_0\in \bar A_{\Theta(t_0)}$ and
consider
\begin{equation}\label{defn_psi} 
 \psi(t):=\bar \phi((\Theta(t_0)/\Omega)+(t-t_0),(\Theta(t_0)/\Omega,x_0)), 
\end{equation}
where
 $\bar\phi$ is the trajectory map  \eqref{defn_barpsi} from  the Kelvin wave the solution $\bar\omega(t)= \mathbf{1}_{\bar A_{\Omega t}}$. Then $\psi$ defined in \eqref{defn_psi} satisfies $\psi(t_0)=x_0$ and \begin{equation*}
	\begin{split}\frac{d}{dt}\psi(t)
&=\bar u((\Theta(t_0)/\Omega)+(t-t_0),\psi(t)). \end{split}
\end{equation*} As a result, we observe 
\begin{equation}\label{psi_same}
 \psi(t)\in \bar A_{\Theta(t_0)+\Omega(t-t_0)}\quad\mbox{for any} \quad  t\geq t_0.
\end{equation}
 By denoting $\phi(t):=\phi(t,(t_0,x_0)),$  
we just need to show 
$$
|\phi(t)-\psi(t)|\leq \gamma_0\cdot\beta\cdot\varepsilon^{1/2}.
$$  
\item First we decompose
\begin{equation}\label{decomp_123}
\begin{split}
	\frac{d}{dt}\left( \phi(t)  - \psi(t) \right)  &= u(t,\phi(t) ) - \bar u((\Theta(t_0)/\Omega)+(t-t_0), \psi(t))\\&=u(t,\phi(t) )- \bar u(\Theta(t)/\Omega,\phi(t))\\
	 &\quad +
\bar u(\Theta(t)/\Omega,\phi(t))-\bar u(\Theta(t_0)/\Omega+(t-t_0),\phi(t))	 \\
&\quad +\bar u(\Theta(t_0)/\Omega+(t-t_0),\phi(t))	     - \bar u((\Theta(t_0)/\Omega)+(t-t_0), \psi(t))\\
& =:I(t)+II(t)+III(t).
\end{split}
\end{equation}
From the stability assumption  \eqref{ass_st_proof},  we have 
\begin{equation*}
	\begin{split}
	|I(t)|&\leq \hat C\varepsilon^{1/2}.
	\end{split}
\end{equation*}  
For $II(t)$, 
we first find $\tilde\Theta(t)\in\mathbb{R}$ satisfying
\begin{equation}\label{defn_til_the}
 \tilde\Theta(t)= \Theta(t)+\frac{2\pi}{m}\cdot k\quad\mbox{for some integer}\quad k \quad\mbox{and}\quad
\left(\tilde\Theta(t ) - (\Theta(t_0) + \Omega(t-t_0)) \right) \in\mathbb{T}.
\end{equation}
We observe that
$\bar u=u^{m,\beta}$ is time-periodic of period
$
\frac{2\pi}{m\Omega}
$
 and is
 Lipschitz  (in space-time) of $\bar u=u^{m,\beta}$ where the Lipschitz norm is uniformly bounded when $\beta>0$ is sufficiently small. Let's denote the Lipschitz constant by $C_{Lip}=C_{Lip}(m)>0$. Then we get
\begin{equation*}
	\begin{split}
 	|II(t)|&=
|\bar u( \tilde\Theta(t)/\Omega,\phi(t))-\bar u(\Theta(t_0)/\Omega+(t-t_0),\phi(t))	|
\leq \frac{C_{Lip}}{\Omega} |\tilde\Theta(t ) - (\Theta(t_0) + \Omega(t-t_0))   | 
\\
&= \frac{C_{Lip}}{\Omega} |\mathcal{T}[\tilde\Theta(t ) - (\Theta(t_0) + \Omega(t-t_0))   ]| = \frac{C_{Lip}}{\Omega} |\mathcal{T}[ \Theta(t ) - (\Theta(t_0) + \Omega(t-t_0))   ]| \\
&\leq 2\frac{C_{Lip}}{\Omega^*}\cdot \left(   C_0\cdot \varepsilon^{1/2} \right),
	\end{split}
\end{equation*} where we used the assumption \eqref{assu_boot_cla} in the last inequality.   
For $III(t)$, we simply have
$$
|III(t)|\leq C_{Lip}|\phi(t)-\psi(t)|.
$$
Thus we have, for $t\in[t_0,t^*]$, \begin{equation*}
\begin{split}
	\frac{d}{dt}|\phi(t) - \psi(t)|    
	& \le C_{Lip}|\phi(t)-\psi(t)| + C_3 \varep^{1/2},
\end{split}
\end{equation*} where $C_3>0$ is some constant (depending only on $m$).
With Gronwall's inequality, we deduce when $t\in[t_0,t^*]\subset [t_0,t_0+{c_0\beta} ]$ that \begin{equation*}
\begin{split}
	|\phi(t)-\psi(t)|& \le e^{C_{Lip}{c_0\beta}}\cdot\int_{t_0}^{t_0+{c_0\beta}}    
C_3 \varep^{1/2} ds \le   \left( e^{ C_{Lip}{c_0\beta}}C_3\right)\cdot {c_0\beta}\cdot \varep^{1/2}.  
\end{split}
\end{equation*} 
We just take a small constant ${c_0}>0$ satisfying
  \begin{equation*}\label{defn_alp_0}
		 \left( e^{ C_{Lip}{c_0\beta_0} }C_3\right)\cdot {c_0} \leq \gamma_0
\end{equation*} (see  \eqref{defn_gam} for $\gamma_0$), which gives
\begin{equation*}
\begin{split}
	|\phi(t)-\psi(t)|& \le  \gamma_0\cdot { \beta}\cdot \varep^{1/2}.  
\end{split}
\end{equation*} 

Thanks to \eqref{psi_same}, we have proved \eqref{pill} for any $t\in[t_0,t^*]\subset[t_0, t_0+{c_0\beta}]$.
	  Now we are ready to show the \textbf{Bootstrap claim} \eqref{hyp_boot} for 
	 $t\in[t_0,t^*]\subset[t_0, t_0+{c_0\beta}]$.
\item To prove, we simply
	denote   $$ E_t:=\phi(t,(t_0,\bar A_{\Theta(t_0)})),\quad  t\geq t_0,
	$$ For instance, we observe $
	E_{t_0}=\phi(t_0,(t_0,\bar A_{\Theta(t_0)}))=\bar A_{\Theta(t_0)},
	$ and \eqref{pill} gives 
$$
E_{t}\subset 
\bar A_{\Theta(t_0)+\Omega(t-t_0)}^{{\gamma_0\cdot\beta}\varepsilon^{1/2}}, \quad \forall t\in[t_0,t^*]\subset [t_0,t_0+c_0\beta].$$

Towards	a contradiction, suppose that the property \eqref{hyp_boot} on the interval $[t_0,t^*]\subset [t_0,t_0+{c_0\beta}]$ fails, \textit{i.e.} 
\begin{equation}\label{hyp_contra}  
		\begin{split}
		\mbox{there is some}\quad t'\in[t_0,t^*]\quad \mbox{satisfying}\quad		|\mathcal{T}[\Theta(t' ) - (\Theta(t_0) + \Omega(t'-t_0))]   | > \frac 1 2 C_0\cdot \varepsilon^{1/2}. 	 		\end{split}
	\end{equation}
	   From now one, we will show $$\|\omega(t')-\bar \omega_{\Theta(t')}\|_{L^1}\geq 2\varepsilon,$$ which gives a contradiction to the stability assumption \eqref{ass_st_proof}.
We begin the estimate with
 \begin{equation*} \label{decomp}
		\begin{split}
		\|\omega(t')-\bar \omega_{\Theta(t')}\|_{L^1}&\geq 
	\|\omega(t') \mathbf{1}_{E_{t'}}-\bar \omega_{\Theta(t')}\|_{L^1}-\|\omega(t') \mathbf{1}_{(E_{t'})^c}\|_{L^1}=:I(t')-II(t').
 		\end{split}
	\end{equation*}	
	For $II(t')$, we estimate
	 \begin{equation*} 
		\begin{split}
 \|\omega(t') \mathbf{1}_{(E_{t'})^c}\|_{L^1}=\|\omega(t_0) \mathbf{1}_{(E_{t_0})^c}\|_{L^1}=\|\omega(t_0)-\bar\omega_{\Theta(t_0)}\|_{L^1({(E_{t_0})^c})}\leq
 \|\omega(t_0)-\bar\omega_{\Theta(t_0)}\|_{L^1 }
 \leq \varepsilon.
 		\end{split}
	\end{equation*}	
	For $I(t')$, we have
 \begin{equation*}
		\begin{split}
		 	\|\omega(t') \mathbf{1}_{E_{t'}}-\bar \omega_{\Theta(t')}\|_{L^1}  
	&\geq 
	\| \bar \omega_{\Theta(t')}\|_{L^1({(E_{t'})^c})}  \geq  
	\| \bar \omega_{\Theta(t')}\|_{L^1({\left(
\bar A_{\Theta(t_0)+\Omega(t'-t_0)}^{{\gamma_0\cdot\beta}\varepsilon^{1/2}}\right)^c})} \\
	&\geq
		\| \bar \omega_{\Theta(t')}\|_{L^1({\left(
\bar A_{\Theta(t_0)+\Omega(t'-t_0)}\right)^c})}-|
\bar A_{\Theta(t_0)+\Omega(t'-t_0)}^{{\gamma_0\cdot\beta}\varepsilon^{1/2}}\setminus 
\bar A_{\Theta(t_0)+\Omega(t'-t_0)} 
|. 
 		\end{split}
	\end{equation*}	 Then,   by using \eqref{est_bar_eta},    we continue to estimate
	\begin{equation*}
		\begin{split}
		 	\dots 
	&\geq 
|\bar A_{\Theta(t')}\setminus\bar A_{\Theta(t_0)+\Omega(t'-t_0)} |
-C_1\cdot {\gamma_0\cdot\beta}\varepsilon^{1/2}
 \\
&=\frac 1 2 
|\bar A_{\Theta(t')}\triangle\bar A_{\Theta(t_0)+\Omega(t'-t_0)} |
-C_1\cdot {\gamma_0\cdot\beta}\varepsilon^{1/2}
 \\
&\geq \frac 1 2 |\bar I_{\Theta(t')-(\Theta(t_0)+\Omega(t'-t_0))}-\bar I|
-C_1\cdot {\gamma_0\cdot\beta}\varepsilon^{1/2}.
 		\end{split}
	\end{equation*}	
Now we can use \eqref{est_bar_i_alp} of  Lemma \ref{lem:v_st} to get
	\begin{equation*}
		\begin{split}
		 	\dots
		 	&\geq \frac 1 2c_3\beta |\mathcal{T}[{\Theta(t')-(\Theta(t_0)+\Omega(t'-t_0))}]|
-C_1\cdot {\gamma_0\cdot\beta}\varepsilon^{1/2} \geq \frac 1 2c_3\beta \frac 1 2 C_0\varepsilon^{1/2}
-C_1\cdot {\gamma_0\cdot\beta}\varepsilon^{1/2},
 		\end{split}
	\end{equation*}	 where we used the hypothesis \eqref{hyp_contra}  in the last inequality.
	Thanks to the definition of $\gamma_0$ in \eqref{defn_gam},
	we have obtained
		$$
			\dots\geq \frac 1 4c_3\beta \frac 1 2 C_0\varepsilon^{1/2},
	$$ which gives 
	$$
			\|\omega(t')-\bar \omega_{\Theta(t')}\|_{L^1}\geq \frac 1 4c_3\beta \frac 1 2 C_0\varepsilon^{1/2}
-\varepsilon.
	$$
	We make $c_0>0$ smaller than before (if necessary)  to satisfy 
\begin{equation}\label{defn_c_0}
3\sqrt{c_0} \leq \frac 1 4c_3  \frac 1 2 C_0. 
\end{equation}  
By this choice of $c_0>0$,
we get,
whenever $0<\varepsilon\leq c_0\beta^2$,
	$$
			\|\omega(t')-\bar \omega_{\Theta(t')}\|_{L^1}\geq 2\varepsilon,
	$$	which is a contradiction to  \eqref{ass_st_proof}. Hence, the hypothesis
\eqref{hyp_contra} cannot be true, which implies that	
	we have proved \textbf{Bootstrap claim} \eqref{hyp_boot} for $[t_0,t^*]\subset [t_0,t_0+{c_0\beta}]$. 

\item Lastly, we are ready to show 
\textbf{Goal} \eqref{goal_boot} for any $t\in[t_0,t_0+{c_0\beta}]$ since we can extend the interval satisfying \textbf{Goal}  \eqref{goal_boot}  by applying \textbf{Bootstrap claim} \eqref{hyp_boot}  with \textbf{Initial claim} \eqref{init_boot}. Indeed, we know that there is $t^*\in(t_0,t_0+{c_0\beta})$ such that \textbf{Goal} \eqref{goal_boot} on the interval $[t_0,t^*]$ holds by using 
 \textbf{Initial claim} \eqref{init_boot}. Then by  applying \textbf{Bootstrap claim} \eqref{hyp_boot}  on the interval $[t_0,t^*]$, we get
 $$	|\mathcal{T}[\Theta(t ) - (\Theta(t_0) + \Omega(t-t_0)) ]  | \leq \frac 1 2 C_0\cdot \varepsilon^{1/2},\quad\forall t\in[t_0,t^*]. 	$$ Then we use \textbf{Initial claim} \eqref{init_boot} by replacing $t_0$ with $t^*$ so that we  get, for any $t\in[t^*,t^*+\eta]$,
  $$	|\mathcal{T}[\Theta(t ) - (\Theta(t^*) + \Omega(t-t^*))  ] | \leq \frac 1 2 C_0\cdot \varepsilon^{1/2}.$$
 By adding the above two estimates, we get, 
 for any $t\in[t^*,t^*+\eta]$,
  $$	|\mathcal{T}[\Theta(t ) - (\Theta(t_0) + \Omega(t-t_0))  ] | \leq   C_0\cdot \varepsilon^{1/2}.$$ In short, we have obtained  \textbf{Goal} on the extended interval $[t_0,t^*+\eta]$.
By repeating this process, we can get \textbf{Goal} on $[t_0,t^*+n\eta]$ for each $n\geq1$
     until the process eventually covers $[t_0,t_0+{c_0\beta}]$.
 \qedhere  \end{enumerate} 
\end{proof}
\begin{remark}\label{rem:sum_est} 
In Proposition \ref{thm:stable}, we can always take $\Theta(0)=0$ (by assuming $\delta\leq \varepsilon$ if necessary). Then summing the estimate \eqref{est_refined} of Proposition \ref{thm:refined} gives 
\begin{equation*}
			|\mathcal{T}[\Theta(t )- \Omega t]    | \leq C_0\cdot \varepsilon^{1/2}\left(\frac t{{c_0\beta}}+1\right)\quad
			  \mbox{for all}\quad   t\in[0,T).
			 	\end{equation*}

\end{remark}
 \subsection{Proof of Theorem \ref{cor:finite_time}}
 Here we will prove \textit{finite time stability} (Theorem \ref{cor:finite_time}) by using \textit{orbital stability} (Proposition \ref{thm:stable}) with \textit{refined stability} (Proposition \ref{thm:refined}).
For $\beta>0$, we denote  
\begin{equation}\label{defn_m}
\bar m=\bar m(\beta):=\sup_{\theta\in\mathbb{T}}(1+g(\theta))>0.
\end{equation}
The next lemma says that when $\beta>0$ is small enough, the trajectories induced from the Kelvin wave starting  near the wave remain close. 
 \begin{lemma}\label{lem:eta}
(I) 
 For  each $\tau>0$, there exist $\beta'>0$ and $\mu>0$  such that
 if $\beta\in[0,\beta']$, then 
		$$\bigcup_{t\geq 0}\bar\phi(t,(0,\bar A^{\mu}))\subset 	 B_{\bar m +\tau},	$$
		where   $\bar\phi$ is the trajectory of $\bar\omega(t)=\mathbf{1}_{\bar A_{\Omega t}}$ as in \eqref{defn_barpsi}, and $\bar A^{\mu}:=\{\dist(x,\bar A)< \mu\}$ as in \eqref{defn_eta}.\\
(II) There exists $\kappa \in(0,1)$ such that for   any $\tau>0$, there exists $\beta'>0$   such that
if $\beta\in[0,\beta']$, then
 for any $x\in\mathbb{R}^2$ with  $1/2\leq |x|\leq 1+\kappa$, we get, for any $t\geq 0$, 
  		$$ \bar\phi(t,(0,x))\in \left( 	 { B_{|x|+\tau}}\setminus \overline{B_{|x|-\tau}}\right).$$
 
 \end{lemma} 
 \begin{proof}
 The first statement just follows from  the second statement. The second statement simply follows the facts that 
  $$\psi^{m,\beta}\to \psi^{m,\beta}|_{\beta=0}\quad\mbox{in}\quad C^{0}\mbox{-norm on} \quad  \overline{B_2} \quad \mbox{as}\quad \beta\to 0
  $$(\textit{e.g.} see   \cite{HMV}),
where $\psi^{m,\beta}$ is the relative stream defined in \eqref{eq:psi-mbt},
and that all the level sets of $\psi^{m,\beta}|_{\beta=0}$ are circles.
Indeed, we recall
 $$ -\partial_r\psi^{m,\beta}|_{\beta=0}(r,\theta)=-\partial_r(G \mathbf{1}_{B_1}+\frac 1 2 \frac{m-1}{2m}r^2)=\begin{cases}& \left(\frac 1 2 -\frac{m-1}{2m}\right)r,\quad r\leq 1,\\
 & \left(\frac 1 {2r^2} -\frac{m-1}{2m}\right)r,\quad r>1\end{cases},
$$ which gives  $$\inf_{r\in[1/3,1+2\kappa] }\left(-\partial_r\psi^{m,\beta}|_{\beta=0}(r,\theta)\right)\geq c>0$$ for some $c=c(m)>0$ and  for some small $\kappa=\kappa(m)>0$.
           We consider any small $\tau>0$ such that 
    $[(1/2)-\tau, 1+\kappa+\tau]\subset [1/3,1+2\kappa]\subset[0,2]$.
Denote $$  \bar\psi:=\psi^{m,\beta},\quad  \hat\psi:=\psi^{m,\beta}|_{\beta=0}.$$ 
For any given point 
$x\in\mathbb{R}^2$ satisfying $1/2\leq |x|\leq 1+\kappa$, we take any points $y', y''\in\mathbb{R}^2$ such that $|y'|=|x|-\tau$ and $|y''|=|x|+\tau$.
We observe that $\hat\psi$ is radially symmetric and 
$$\hat\psi(y')-\hat\psi(x)\geq   c\tau,\quad \hat\psi(x)-\hat\psi(y'')\geq   c\tau.$$
Then, by using the uniform convergence $\bar\psi\to\hat\psi$, we can take $\beta>0$ small enough to get 
$$\sup_{|y|=|y'|}|\bar\psi(y)-\hat\psi(y')|\leq \frac{c\tau}{8},\quad \sup_{|y|=|y''|}|\bar\psi(y)-\hat\psi(y'')|\leq \frac{c\tau}{8},
 \quad  |\bar\psi(x)-\hat\psi(x)|\leq \frac{c\tau}{8},
$$
Thus we get
$$
\sup_{|y|=|x|+\tau}\bar\psi(y)<\bar\psi(x)<\inf_{|y|=|x|-\tau}\bar\psi(y),
$$ which implies that
the connected component of the  level set of $\bar\psi$ containing the point $x$ completely  lies on the annulus 
$
{ B_{|x|+\tau}}\setminus \overline{B_{|x|-\tau}}.
$ 
Since the trajectory $\bar\psi(t,(0,x)),\,t\geq 0$ should lie on the connected component  of the  level set of $\bar\psi$ containing the point $x$, we are done. \end{proof}
 To prove 
Theorem \ref{cor:finite_time}, we just need the first statement of the above lemma while the second one  will be used in the next subsection when proving Theorem \ref{thm:instability}. 
 
 \begin{proof}[Proof of Theorem \ref{cor:finite_time}]
\begin{enumerate}
\item We first borrow
the constant $\beta_0>0$  from Proposition \ref{thm:refined}, and consider small $\beta_2\in(0,\beta_0]$ satisfying
\begin{equation}\label{mu_cond2}
\bar{m}(\beta)\leq 1+2\beta<\bar r\quad \mbox{and}\quad
\bar{n}(\beta)\geq 1-2\beta,\quad\forall \beta\in(0,\beta_2],
\end{equation} where $ \bar m=\bar m(\beta)\in(0,\bar r)$ is defined in \eqref{defn_m},	
 where   $\bar r >0 $ is the   constant  required in 
\eqref{cond_evol} for orbital stability of Proposition \ref{thm:stable},
and where  $\bar{n}(\beta)$ is the minimum radius of the Kelvin wave: $\bar n=\bar{n}(\beta):=\inf_{\theta\in\mathbb{T}}(1+g(\theta))>0$. 
Then we    simply set  
\begin{equation}\label{defn_tau}
\tau:=\frac{\bar r -(1+2\beta_2)}{2}>0,
\end{equation}
 	 and take the two constants $\beta'=\beta'(\tau)>0$ and  $\mu=\mu(\tau)>0$ from $(I)$ of Lemma \ref{lem:eta}.
 	 Let $\beta_1>0$ small enough to have
 	 \begin{equation}\label{mu_cond1}
\beta_1\leq\min(\beta_2,\beta')\quad\mbox{and}\quad  2\beta_1\leq \frac{\mu}{2}.
\end{equation}

 We also set 
\begin{equation}\label{defn_r'}
r':=1+\frac{\mu}{2}>1.
\end{equation} We may assume 
\begin{equation}\label{assm_r'_bar_r}
r'<\bar r
\end{equation} (by redefining $\mu>0$ if necessary).
Then,  
	  $(I)$ of Lemma \ref{lem:eta} says that 
	 \begin{equation}\label{conseq_eta}
	 \cup_{t\geq 0}\bar\phi(t,(0,\bar A^{\mu}))\subset B_{\bar m +\tau }.
	 \end{equation}

\item 
 From now on, we fix any $\beta\in(0,\beta_1]$.
Let $T,\varepsilon'>0$.
We define $C_4=C_4(T,\beta)>0$   by
$$ C_4(T,\beta):=C_0\cdot \left(\frac T{{c_0\beta}}+1\right),$$ where $c_0, C_0>0$ are the constants from Proposition \ref{thm:refined}.		
We set $C_5=C_5(T,\beta)>0$  by 
\begin{equation}\label{defn_C5}
 C_5(T,\beta):=\hat C+\frac{2\cdot C_{Lip}}{\Omega^*}\cdot C_4({T,\beta}),
 \end{equation} where $\hat C>0$ comes from \eqref{defn_hat_C}. We  consider any small   $\varepsilon\in(0,c_0\beta^2]$ satisfying
  \begin{equation}\label{est_C5}
	\left( e^{ C_{Lip}T}C_5\right)\cdot T\cdot \varep^{1/2}\leq \frac{\tau}{2}\quad\mbox{and}\quad
 \varepsilon+2c_5C_4\varepsilon^{1/2}\leq \varepsilon',
\end{equation} where $c_5>0$ is the constant of \eqref{est_v_st_der2} in Lemma \ref{lem:v_st}.
 Then,  Proposition \ref{thm:stable} together with  summing the estimate \eqref{est_refined} of Proposition \ref{thm:refined} says that
there is $\delta':=\delta(\beta,\varepsilon)>0$, where 
$\delta(\beta,\varepsilon)$ is the constant from Proposition \ref{thm:stable},
such that
if a $m$-fold symmetric initial data $\omega_0= \mathbf{1}_{A_0}$ satisfies
$$ \|\omega_0-\bar\omega\|_{L^1}\leq \delta' $$
and
 if the corresponding solution $\omega(t)= \mathbf{1}_{A(t)}$ satisfies
\begin{equation}\label{radi_const}
\mbox{\textbf{Range hypothesis for $t'$:}}\quad \cup_{t\in[0,t')}A(t)\subset B_{\bar r}
\end{equation}  for some $t'\in(0,T]$,
     then  there exists a function $\Theta:[0,t')\to\mathbb{T}$ such that 
\begin{equation}\label{ass_st_pf2}
\sup_{t\in[0,t')}\|\omega(t)-\bar{\omega}_{\Theta(t)}\|_{L^1(\mathbb{R}^2)}\leq \varepsilon
\end{equation} and
\begin{equation}\label{est_finite}
			\sup_{t\in[0,t')}|\mathcal{T}[\Theta(t )- \Omega t]    | \leq  C_4\varepsilon^{1/2}. 
			 	\end{equation} 
			 	{(\textit{e.g.} see Remark \ref{rem:sum_est})}.
		 
\item
 From now on, we
  consider
  any $m$-fold symmetric initial data $\omega_0= \mathbf{1}_{A_0}$ satisfying
the initial condition \eqref{ass_st_cor}.
 We will show that  \textbf{Range hypothesis} \eqref{radi_const} is valid for $t'=T$. First,    due to the initial assumption with \eqref{assm_r'_bar_r}, the hypothesis is true for some $t'>0$ since the flow speed is uniformly bounded.
For a contradiction, let's suppose that the hypothesis fails for $t'=T$.
Then, there exists  some moment $T_0\in(0,T)$ 
  such that 
\begin{equation}\label{contra_ass}
\mbox{ \textbf{Range hypothesis} holds for $t'=T_0$ while the hypothesis fails for every
  $t'>T_0$.}
\end{equation}
\item We note that since the hypothesis is true for  $t'=T_0$,  the estimates
\eqref{ass_st_pf2} and \eqref{est_finite} hold for $t'=T_0$.
		 For $x\in\mathbb{R}^2$, 
		we denote $\phi(t):=\phi(t,(0,x))$ from 
		\eqref{trajec_pertu} and 
		$\bar\phi(t):=\bar\phi(t,(0,x))$ from \eqref{defn_barpsi}. In the computations below, we consider $t\in[0,T_0)$.
		Similarly in \eqref{decomp_123}, we compute
		\begin{equation}\label{decomp_re}
\begin{split}
	\frac{d}{dt}\left( \phi(t)  - \bar\phi(t) \right)  = u(t,\phi(t) ) - \bar u(t,\bar\phi(t))=&u(t,\phi(t) )- \bar u(\Theta(t)/\Omega,\phi(t))\\
	 &\quad +
\bar u(\Theta(t)/\Omega,\phi(t))- \bar u(t,\phi(t))  \\
&\quad +   \bar u(t,\phi(t)) - \bar u(t,\bar\phi(t))\\
& =:I(t)+II(t)+III(t).
\end{split}
\end{equation}

 Then
the estimate \eqref{ass_st_pf2},  we have 
\begin{equation*}
	\begin{split}
	|I(t)|&\leq \hat C\varepsilon^{1/2}.
	\end{split}
\end{equation*} 
 For $II(t)$, 
as in \eqref{defn_til_the}, we first find $\tilde{\Theta}(t)$ by
\begin{equation*}
 \tilde\Theta(t)= \Theta(t)+\frac{2\pi}{m}\cdot k\quad\mbox{for some integer}\quad k \quad\mbox{and}\quad
\left(\tilde\Theta(t ) -  \Omega t \right) \in\mathbb{T}.
\end{equation*} Then, by using time-periodicity and (space-time) Lipschitz continuity of $\bar u$, we get
 \begin{equation*}
	\begin{split}
	|II(t)|&=|
\bar u(\tilde\Theta(t)/\Omega,\phi(t))- \bar u(t,\phi(t))|
\leq \frac{2\cdot C_{Lip}}{\Omega^*} |\tilde\Theta(t ) - \Omega t   | 
\\
&
=\frac{2\cdot C_{Lip}}{\Omega^*} |\mathcal{T}[\tilde\Theta(t ) - \Omega t]   | 
=\frac{2\cdot C_{Lip}}{\Omega^*} |\mathcal{T}[ \Theta(t ) - \Omega t]   | 
\\
&\leq  \frac{2\cdot C_{Lip}}{\Omega^*}\cdot C_4\varepsilon^{1/2},
	\end{split}
\end{equation*} 
where we used the estimate \eqref{est_finite} in the last inequality.
For $III(t)$, we get
$$
|III(t)|\leq C_{Lip}|\phi(t)-\bar\phi(t)|,
$$ which gives,  \begin{equation*}
\begin{split}
	\frac{d}{dt}|\phi(t) -\bar\phi(t)|    
	& \le C_{Lip}|\phi(t)-\bar\phi(t)| + C_5 \varep^{1/2}.
\end{split}
\end{equation*} where $C_5>0$ was already  defined in \eqref{defn_C5}.
With Gronwall's inequality and with the smallness assumption (on $\varepsilon$) \eqref{est_C5}, we get  \begin{equation*}
\begin{split}
	|\phi(t)-\bar\phi(t)|
	& \le   \left( e^{ C_{Lip}T}C_5\right)\cdot T\cdot \varep^{1/2}\leq\frac{\tau}{2}.  
\end{split}
\end{equation*} 
 Together with the fact \eqref{conseq_eta} and the definition \eqref{defn_tau} of $\tau$, the argument above implies  \begin{equation*}
	 \cup_{t\in[0,T_0)}\phi(t,(0,\bar A^{\mu}))\subset B_{(\bar m +\tau) +(\tau/2) }.
	 \end{equation*} which gives  
	 \begin{equation*}
	 \cup_{t\in[0,T_0]}\phi(t,(0,\bar A^{\mu}))\subset \overline{B_{(\bar m +\tau) +(\tau/2) }}.
	 \end{equation*} 
On the other hand, the definition \eqref{defn_r'} of $r'$ together with 
  \eqref{mu_cond1} and \eqref{mu_cond2} implies
$$
B_{r'}\subset B_{\bar{n}+\mu}\subset \bar{A}^\mu.
$$
   Thus, the assumption  $A_0\subset B_{r'}$ gives
$$
\cup_{t\in[0,T_0]}A(t)=\cup_{t\in[0,T_0]}\phi(t,(0,A_0))
\subset \cup_{t\in[0,T_0]}\phi(t,(0,\bar A^{\mu}))\subset \overline{B_{(\bar m +\tau) +(\tau/2) }}.
$$ 
On the other hand, we observe $$
{(\bar m +\tau) +(\tau/2) }<\bar r$$ thanks to \eqref{defn_tau}, \eqref{mu_cond2}.
By recalling that  the flow speed is bounded, there should exist some moment $T_1>T_0$ such that \textbf{Range hypothesis} is true for $t'=T_1$, which contradicts the assumption \eqref{contra_ass}.
Hence, \textbf{Range hypothesis} \eqref{radi_const} for $t'=T$ is valid.
As a result,   we obtain the   estimates  \eqref{est_finite},  \eqref{ass_st_pf2} for $t'=T$.
\item Lastly,     by   the estimate \eqref{est_v_st_der2} of Lemma \ref{lem:v_st} and by \eqref{ass_st_pf2}, \eqref{est_finite} for $t'=T$,  we get, for any $t\in[0,T)$,
  \begin{equation*}
\begin{split}
	\|\omega(t)-\bar \omega_{\Omega t}\|_{L^1}&\leq
 \|\omega(t)-\bar \omega_{\Theta(t)}\|_{L^1}+ \|\bar \omega_{\Theta(t)}-\bar \omega_{\Omega t}\|_{L^1}\\& 
\leq 
 \varepsilon+2|\bar A_{\Theta(t)}\setminus \bar A_{\Omega t}|
 \leq 
 \varepsilon+2c_5|\mathcal{T}[\Theta(t)-\Omega t]| \leq 
 \varepsilon+2c_5C_4\varepsilon^{1/2}.
\end{split}
\end{equation*}
By the smallness assumption \eqref{est_C5} (on $\varepsilon$),  we get
the  stability ($\eqref{ass_st_cor}^{\delta'}\Rightarrow\eqref{est_cor}^{\varepsilon'}$).
 It finishes the proof of Theorem \ref{cor:finite_time}. \qedhere

\end{enumerate}

\end{proof}

\subsection{Proof of Theorem \ref{thm:instability}}
Now we are ready to  prove \textit{perimeter growth theorem} (Theorem \ref{thm:instability}) by using \textit{finite time stability} (Theorem \ref{cor:finite_time}) and $(II)$ of Lemma \ref{lem:eta}.
\begin{proof}[Proof of Theorem \ref{thm:instability}]
\begin{enumerate}
\item We recall that the angular velocity $\hat u^\theta$ of $\hat{u} := K * \mathbf{1}_{B_1}$ is
 $$\hat u^\theta(r)=
 \begin{cases} \frac {1} 2,\quad  & r\leq 1,\\    \frac {1}{2r^2},\quad & r>1\end{cases}.
 $$
 \item 
We   borrow the constants $\beta_1>0, \,r'>1$ from  Theorem \ref{cor:finite_time} and set $\mu:=r'-1>0$. We also take the constant $\kappa>0$ from $(II)$ of Lemma \ref{lem:eta}.
We  set $r_1:=1$ and $$r_2:=1+
 \frac{\min\{\mu,\kappa\}}{2}.
 $$ We note $
r_2<r'$ and $r_2<1+\kappa.$ Then we  
 consider any constant $\tau\in(0,1/4]$ satisfying
$$\tau\leq \frac{\min\{\mu,\kappa\}}{20},$$ 
 which will be chosen again to be small during the proof. 
 \item We denote the intervals
 $$
 I_i=[r_i-2\tau,r_i+2\tau],\quad  I'_i=[r_i-\tau,r_i+\tau]
 $$  for $i=1,2$. 
  We set 
 $$U^1:=\left(\inf_{r\in I_1}\hat u^\theta(r)\right)>0\quad\mbox{and}\quad
 U^2:=\left(\sup_{r\in I_2}\hat u^\theta(r)\right)>0.$$
Since $
\hat u^\theta(r_1)> \hat u^\theta(r_2)
$ and $\hat u^\theta$ is continuous, we can
assume  $U^1-U^2>0$ by making $\tau>0$ smaller than before (if necessary).
By fixing such a constant $\tau>0$, we take $\beta'=\beta'(\tau)>0$ from $(II)$ of Lemma \ref{lem:eta}.
  We also denote 
 $$\bar U^1:= U^1-\frac {U^1-U^2} 4 \quad\mbox{and}\quad
  \bar U^2:= U^1+\frac {U^1-U^2} 4,$$ and note that $\Delta \bar U:=\bar U^1-\bar U^2>0.$
 
\item  Let $M>0$ and $\delta>0$. We take any large $T>0$ such that 
 \begin{equation}\label{defn_T}
\left(T \Delta \bar U -2\pi\right) > 2M. 
 \end{equation}
  Let $\varepsilon>0$ be small enough 
to satisfy
\begin{equation}\label{inf_small2}\hat C (\varepsilon)^{1/2}\leq \frac{U^1-U^2}{8}, \end{equation} where $\hat C>0$ comes from \eqref{defn_hat_C}, and
 \begin{equation}\label{small_ep'}
	\left( e^{ C_{Lip}T} \right)\cdot T\cdot \hat  C(\varepsilon)^{1/2}\leq\tau.
\end{equation}


\item 
  From now on, we consider a sufficiently small
   $\beta\in(0,\min\{\beta_1,\beta'\}]$
satisfying the following:
 \begin{enumerate}
 \item The perimeter 
 of $\partial\bar A$ is smaller than $10$.
  \item $\partial \bar A\cap \{r=1\}\neq \emptyset$
 \item The velocity $\bar u=K*\mathbf{1}_{\bar A}$ for the Kelvin wave with parameter $\beta>0$ is close enough to  the velocity $\hat u$ for the circular patch  in the sense that
\begin{equation}\label{inf_small1}
\|\bar u-\hat u\|_{L^\infty}\leq \frac{U^1-U^2}{8}.
\end{equation} 
 \end{enumerate}

   We may assume   $\delta>0$ small enough to satisfy  $\delta\leq\delta(m,\beta,\varepsilon,T)$, where $\delta(m,\beta,\varepsilon,T)>0$ is the constant  from Theorem \ref{cor:finite_time}. 
 
 \item  We take any initial data $\mathbf{1}_{A_0}$ with the following properties:
  \begin{enumerate}
 \item $A_0$ is an open $m$-fold symmetric set with $C^\infty$-smooth connected boundary $\partial A_0$.
 \item The perimeter 
 of $\partial A_0$ is smaller than $20$.
 \item 
$A_0\subset B_{r'}$ 
 and $
\|\omega_0-\bar\omega\|_{L^1(\mathbb{R}^2)}\leq\delta.
$
\item  For each $i=1,2$, $\exists$ a point $x^i=(r_i\cos\tht_i,r_i\sin\tht_i) \in \partial  A_0$ satisfying $|\tht_i|\le \frac{2\pi}{m}$. 
 \end{enumerate}
  Then, Theorem \ref{cor:finite_time} implies that 
 the perturbed solution $\omega(t)=1_{A_t}$ satisfies 
 \begin{equation}\label{ass_st_pf3}
   \sup_{t\in[0,T]}\|\omega(t)-\bar\omega_{\Omega t}\|_{L^1(\mathbb{R}^2)}\leq\varepsilon'.
 \end{equation}
 Set $L_0:[0,1]\to\partial A_0$ be an injective parametrized curve lying on the sector $\{(r,\theta)\,:\,\theta\in\mathbb{T}\}$ satisfying $L_0(0)=x^1, L_0(1)=x^2$, and consider
 $L_t:=\phi(t,(0,L_0)).$ We will show that the length of the curve $L_T$ is larger than $M$, which finishes the proof thanks to $L_T\subset \partial A_T$.
 \item For $i=1,2$, we denote 
$$\phi^i(t):=\phi(t,(0,x^i))\quad\mbox{and}\quad \bar\phi^i(t):=\bar\phi(t,(0,x^i)),$$
where $\phi, \bar\phi$ are the trajectories from the perturbed solution $\omega(t)$ and the Kevin wave solution $\bar\omega(t)=\mathbf{1}_{\bar A_{\Omega t}}$ as in 
\eqref{trajec_pertu}, \eqref{defn_barpsi}, respectively.
 For any interval $I=[a,b]\subset \mathbb{R}_{>0}$, we denote the annulus 
 $$
 R_{I}:=\overline{B_{b}}\setminus B_{a}.
 $$ 
We observe
\begin{equation}\label{est_bar_per}
 \bar\phi^i(t)\in R_{  I'_i},\quad\forall t\geq 0,\quad i=1,2. 
\end{equation}
   by using $(II)$ of  Lemma \ref{lem:eta}.
 \item  We claim
 \begin{equation}\label{claim_i_a}
  \phi^i(t)\in R_{I_i},\quad\forall t\in[0,T]
 \end{equation}
for $i=1,2$. Thanks to \eqref{est_bar_per},  it is enough to show that
$$|\phi^i(t)-\bar\phi^i(t)|\leq \tau,\quad \forall t\in[0,T].$$
 	Similarly in \eqref{decomp_123}, \eqref{decomp_re}, we compute
		\begin{equation*}
\begin{split}
	\frac{d}{dt}\left( \phi(t)  - \bar\phi(t) \right)  = u(t,\phi(t) ) - \bar u(t,\bar\phi(t))=&u(t,\phi(t) )- \bar u(t,\phi(t))\\
	 &\quad +
\bar u(t,\phi(t))- \bar u(t,\bar \phi(t))  \\
& =:I(t)+II(t).
\end{split}
\end{equation*}

From the stability   \eqref{ass_st_pf3},  we have 
\begin{equation*}
	\begin{split}
	|I(t)|&\leq \hat C(\varepsilon')^{1/2}.
	\end{split}
\end{equation*}

For $II(t)$, we 
note Lipschitz continuity of $\bar u$  to get
\begin{equation*}
	\begin{split}
	|II(t)|&
\leq  {C_{Lip}} |\phi(t)-\bar\phi(t)  |. 
	\end{split}
\end{equation*}  Thus 
  we  get, for $t\in[0,T]$, \begin{equation*}
\begin{split}
	\frac{d}{dt}|\phi(t) -\bar\phi(t)|    
	& \le C_{Lip}|\phi(t)-\bar\phi(t)| + \hat C(\varepsilon')^{1/2}.
\end{split}
\end{equation*}  
With Gronwall's inequality, we get  \begin{equation*}
\begin{split}
	|\phi(t)-\bar\phi(t)|
	& \le   \left( e^{ C_{Lip}T} \right)\cdot T\cdot \hat C(\varepsilon')^{1/2}.
\end{split}
\end{equation*} 
From the smallness  assumption \eqref{small_ep'} on $\varepsilon'>0$,
  we obtain the claim \eqref{claim_i_a}.
 
\item
Lastly, we observe 
\begin{equation*}\begin{split}
|u(t, \phi^i(t))-\hat u(\phi^i(t))|\leq |u(t, \phi^i(t))-\bar u(t,\phi^i(t))|+|\bar u(t, \phi^i(t))-\hat u(\phi^i(t))|\leq \frac{U^1-U^2}{4},
\end{split}\end{equation*} where the last inequality follows from \eqref{ass_st_pf3}, \eqref{inf_small2}, \eqref{inf_small1}.
Thus  the above claim \eqref{claim_i_a} implies
  that the angular velocity of $u(t,\phi^1(t))$ is bigger than $\bar U^1$ while that   of $u(t,\phi^2(t))$ is smaller than $\bar U^2$. Thus we simply  observe that the difference between the
 winding number (with respect to the origin) of trajectory $\phi^1(t)$ on $[0,T]$ starting at $x^1$ and the  winding number of trajectory $\phi^2(t)$ starting at $x^2$ is bigger than  
 $$\frac{\left( T\Delta \bar U-2\pi\right)}{2\pi}.$$
 Since  $$\phi^i(t)\in R_{I_i}\subset \mathbb{R}^2\setminus B_{1/2},\quad i=1,2
 $$  on $[0,T]$, our choice  \eqref{defn_T} of $T$ implies that the length of $L(T)
 $ should be larger than 
$  M. $ \end{enumerate} 
The proof is complete. \end{proof}

\subsection*{Acknowledgment}

\noindent  {We would like to thank Junho Choi for allowing the use of Figure \ref{fig:fila}.} KC has been supported by the National Research Foundation of Korea (NRF-2018R1D1A1B07043065). IJ has been supported by the Samsung Science and Technology Foundation under Project Number SSTF-BA2002-04 and the New Faculty Startup Fund from Seoul National University. 

\appendix 

\section{Stability of the Annulus}

In this section, we provide a sketch of the fact that any annulus is a strict local maximum of the energy within a suitable admissible class of patches. Based on this fact, one can derive nonlinear stability and instability results as in the Kelvin wave case. We believe that this is interesting at least for the following reasons: \begin{itemize}
	\item While it is known that monotone decreasing and radial vorticities define nonlinear stable steady states {(\textit{e.g.} see \cite{CL_radial}),} this seems to be a fist instance where nonlinear stability for \textit{non-monotone} radial solution can be obtained. Moreover, long time filamentation can be proved near any annulus using the nonlinear stability. 
	\item It is likely that under certain mass, impulse, and $m$-fold symmetry constraint, there exist at least two strict local maximum of the energy, one given by an $m$-fold symmetric Kelvin wave and the other being an annulus, especially when both of them are sufficiently close to the disc. (Strictly speaking, we do not know the precise range of existence/stability in $\bt$ for Kelvin waves with large $m$.)
\end{itemize}

\subsection{Admissible class and key proposition for the annulus}

For $0< r_1 < r_2$, we consider the annulus  \begin{equation*}
	\begin{split}
		\bar{\omg}_{r_1,r_2} := \mathbf{1}_{[r_1,r_2]}(r). 
	\end{split}
\end{equation*} We shall often omit writing out the subscripts $r_1$ and $r_2$, and define the admissible class of perturbations \begin{equation*}
	\begin{split}
		\calA[\bar{\omg}] := \left\{ \tilde{\omg} = \mathbf{1}_{A} :   \int \tilde{\omg} = \int \bar{\omg},  \int |x|^2 \tilde{\omg} = \int |x|^2 \bar{\omg} \right\}
	\end{split}
\end{equation*} and set \begin{equation*}
	\begin{split}
		\calA^{m} = \calA[\bar{\omg}] \cap \left\{ \tilde{\omg} \mbox{ is } m\mbox{-fold symmetric
		} \right\}.
	\end{split}
\end{equation*} It is interesting to note that, imposing the mass and impulse constraint simultaneously picks out (at most) one annulus. Next, we set \begin{equation*}
	\begin{split}
		\calN_{\varepsilon,D}[\bar{\omg}] := \left\{  \tilde{\omg} = \mathbf{1}_{A} : A\subset D, \nrm{\tilde{\omg}-\bar\omg}_{L^1} < \varepsilon \right\}.
	\end{split}
\end{equation*}
Let us now state our key proposition. 
\begin{proposition}\label{prop:ann}
	For any $0<r_1<r_2$, there exist $m\ge2$, $\bar{r}>1$, $\varepsilon_0 > 0$, and $c_0>0$ such that \begin{equation*}
		\begin{split}
			E[\bar\omg_{r_1,r_2}] - E[\omg] \ge c_0 \nrm{\bar\omg_{r_1,r_2} - \omg}_{L^1}^2
		\end{split}
	\end{equation*} for any $\omg \in \calA^{m} \cap \calN_{\varepsilon,B_{\bar{r}r_2}}[\bar{\omg}_{r_1,r_2}]$ with $0<\varepsilon<\varepsilon_0$. 
\end{proposition}

All of the constants $m, \bar{r}, \varepsilon_0$, and $c_0$ depend on $r_1$ and $r_2$ in a rather complicated way. The rest of this section is devoted to the proof of the above proposition. We omit the details as the arguments are parallel to the case of the Kelvin waves. 

\subsection{Relative stream function}

We now modify the stream function of $\bar{\omg}$ in a way that it vanishes on the boundary of the annulus. We recall that for any radial vorticity $\bar{\omg}$, $\bar{G} := G[\bar\omg] = \frac{1}{2\pi} \ln \frac{1}{|x|} * \bar{\omg}$ is given by \begin{equation*}
	\begin{split}
		-\rd_r \bar{G} = \frac{1}{r} \int_0^r s \bar{\omg}(s) ds. 
	\end{split}
\end{equation*} Indeed, using the above formula it is immediate to see that $\lap \bar{G} = (\rd_{rr}+ \frac{\rd_r}{r} )\bar{G} = - \bar{\omg}$. We see that \begin{equation*}
	\begin{split}
		-\rd_r \bar{G} (r) = \begin{cases}
			0 & r\le r_1, \\
			r/2 - r_1^2/(2r) & r_1 < r \le r_2, \\
			(r_2^2-r_1^2)/(2r) & r_2 < r. 
		\end{cases}
	\end{split}
\end{equation*} We have that $\bar{G}(0) = \int_{r_1}^{r_2} r\ln \frac{1}{r} dr$ and $\bar{G}(r)$ is monotone decreasing in $r$. We claim that there exists a unique pair of constants $C_0,C_1$ such that the \textit{relative stream function} defined by \begin{equation}\label{eq:rel-stream-ann0}
	\begin{split}
		\bar{\psi} = \bar{G} + C_0 + C_1 r^2 
	\end{split}
\end{equation} satisfies \begin{equation*}
	\begin{split}
		\bar{\psi}(r_1) = \bar{\psi}(r_2) = 0. 
	\end{split}
\end{equation*} 
The unique choice is given by \begin{equation*}\label{eq:rel-stream-ann}
	\begin{split}
		\bar{\psi} = \bar{G} - \bar{G}(0) - C_1r_1^2 + C_1 r^2 , \qquad C_1 := \frac{1}{4} - \frac{r_1^2}{2(r_2^2-r_1^2)}\ln \frac{r_2}{r_1} > 0.
	\end{split}
\end{equation*} We note that \begin{equation*}
	\begin{split}
		\rd_r \bar{\psi}(r_1) = 2C_1r_1 >0, \qquad \rd_r\bar{\psi}(r_2) = 2C_1r_2 + \bar{G}'(r_2) = - \frac{r_1^2r_2}{r_2^2-r_1^2} \ln \frac{r_2}{r_1} + \frac{r_1^2}{2r_2}<0 
	\end{split}
\end{equation*} for any $0<r_1<r_2$. See Figure \ref{fig:annstream} for a plot of $\bar{G}$ and $\bar{\psi}$ in the case $r_1=1/2$ and $r_2=1$. Note that there is a critical radius $r^* > r_2$ (depending on $r_1$ and $r_2$) such that $\bar{\psi}(r^*)=0$ and $\bar\psi (r) < 0$ for $r_2<r<r^*$. This determines $\bar{r}$ in Proposition \ref{prop:ann}; we need to take $1< \bar{r} < r^*/r_2$.

\begin{figure}
	\centering
	\includegraphics[scale=0.5]{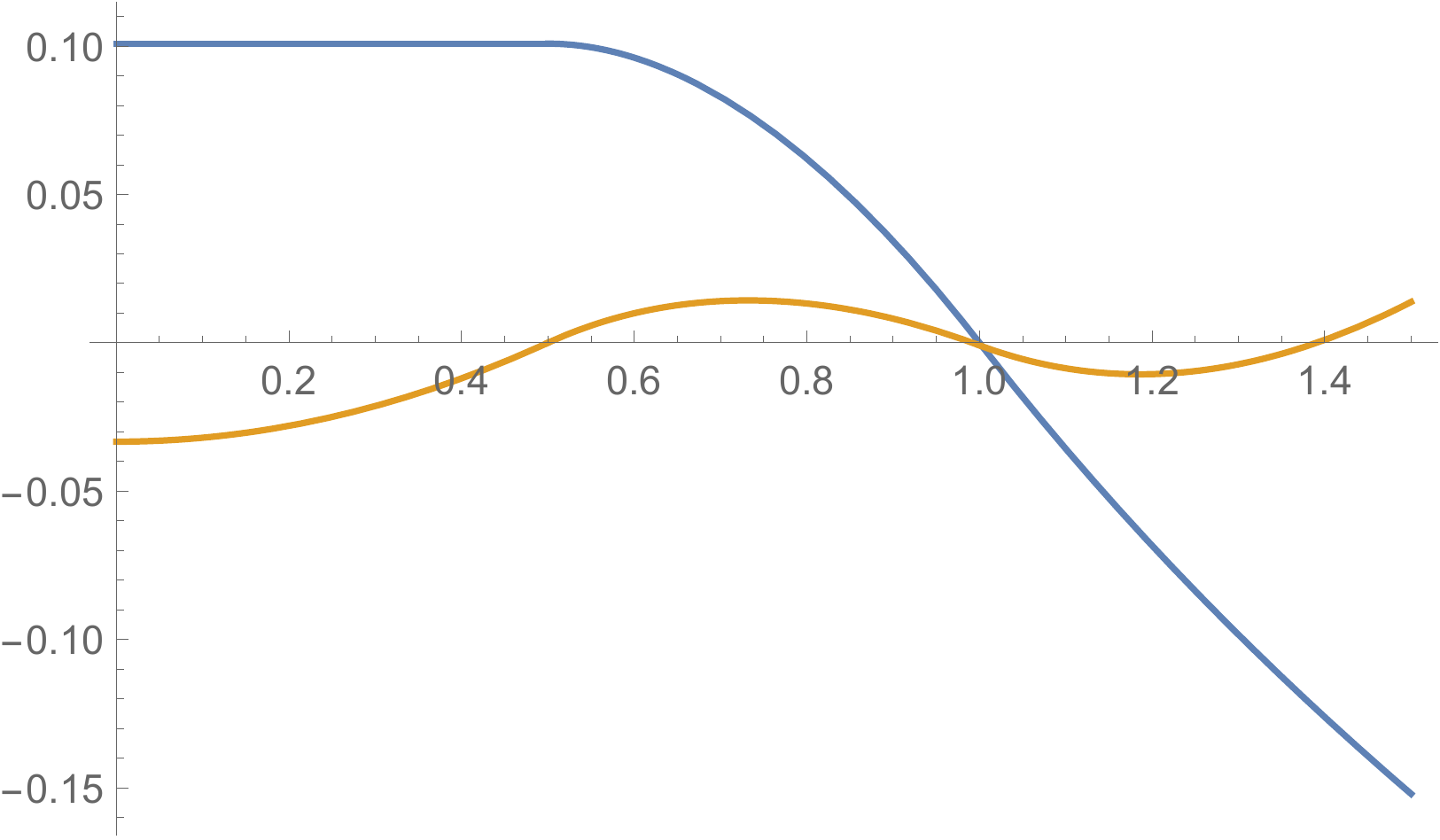}  
	\caption{Stream functions $\bar{G}$ and $\bar\psi$ for the annulus} \label{fig:annstream}
\end{figure}

\subsection{Graph type perturbation}

We fix some $\bar{\omg}=\bar\omg_{r_1,r_2}$ and  consider graph type perturbations, which are described by a pair of functions defined on $S^1$; given $(h_1(\tht), h_2(\tht))$ which are assumed to be sufficiently small (depending on $r_1$ and $r_2-r_1$) in the $C^1$ norm, we set \begin{equation*}
	\begin{split}
		\tilde{\omg}_{h_1,h_2} := \mathbf{1}_{\tilde{A}_{h_1,h_2}}, \qquad \tilde{A}_{h_1, h_2} = \{  (r,\tht):   r_1 + h_1(\tht) < r < r_2 + h_2(\tht)  \}. 
	\end{split}
\end{equation*} 
We easily compute that, with notation $\nrm{h_i}^2 := \int_{S^1} |h_i|^2 d\tht$,  \begin{itemize}
	\item Mass: \begin{equation*}
		\begin{split}
			\int \tilde{\omg} dx  = \int \bar{\omg} dx -r_1 \int h_1 d\tht + r_2 \int h_2 d\tht + \frac12 \int h_2^2 d\tht - \frac12 \int h_1^2 d\tht . 
		\end{split}
	\end{equation*}  
	\item Impulse: \begin{equation*}
		\begin{split}
			\int |x|^{2} \tilde{\omg} dx   = \int |x|^{2} \bar{\omg} dx  + r_2^{3} \int h_2 d\tht + \frac32 r_2^{2} \int h_2^2  - r_1^{3} \int h_1 d\tht -  \frac32 r_1^{2} \int h_1^2 + O( \nrm{h_1}^3 + \nrm{h_2}^3 ).  
		\end{split}
	\end{equation*} 
\end{itemize}
Based on these computations, we immediately see that the requirement $\tilde{\omg} \in \calA[\bar\omg]$ forces that  \begin{equation}\label{eq:small-mass}
	\begin{split}
		\int h_2 d\tht , \quad  \int h_1 d\tht = O(\nrm{h}^2),
	\end{split}
\end{equation} where $\nrm{h}^2:=\nrm{h_1}^2+\nrm{h_2}^2$. This \textit{small mass condition} will be used frequently in the following. 

\subsection{Reduction to graph type perturbation} Given a fixed annulus $\bar\omg$ and a (general) patch perturbation $\omg^*$ belonging to the admissible class $\calA^{m} \cap \calN_{\varepsilon,B_{\bar{r}r_2}}$, we need to find a graph type perturbation $\tilde\omg$ which satisfies \begin{equation*}
	\begin{split}
		E[\omg^*] - E[\tilde\omg] \le 0, \qquad E[\tilde\omg] - E[\bar\omg] \le 0 
	\end{split}
\end{equation*} and still belonging to the admissible class. The proof of the second inequality is the goal of the next section. For the first inequality, having $\omg^*$ close to $\bar\omg$ {in $L^1$} implies that the function $G[\omg^*] + C_0 + C_1r^2$ is close to $\bar\psi$ in the $C^{1,\alp}$ topology with any $0<\alp<1$, where $C_0$ and $C_1$ are the constants from \eqref{eq:rel-stream-ann0}. Then, there is a unique way to slightly perturb the constants $C_0$ and $C_1$ to $C_0'$ and $C_1'$ respectively, so that if we define $\tilde{\omg}$ to be the patch supported on the {inner} 
component of the set $ \{ \tilde{G} > 0 \}$ where \begin{equation*}
	\begin{split}
		\tilde{G} := G[\omg^*] + C_0' + C_1'r^2,
	\end{split}
\end{equation*} then $\tilde{\omg}$ belongs to the admissible class. This can be proved using a determinant computation arising from matching the mass and impulse simultaneously. Then, proving $E[\omg^*] - E[\tilde\omg] \le 0$ is straightforward, using that $\rd_r\bar\psi (r_1)>0 > \rd_r\bar\psi(r_2)$.

\subsection{Energy difference for graph type perturbation} Finally, we may assume that $\tilde{\omg}$ is a graph type perturbation and write 
\begin{equation*}
	\begin{split}
		{E}[\tilde{\omg}] - E[\bar\omg] = \brk{\tilde{\omg} - \bar\omg, \bar{G} } + \frac12 \brk{\tilde{\omg} - \bar\omg, G(\tilde{\omg} - \bar\omg)} = I+II. 
	\end{split}
\end{equation*}

\medskip

\noindent \underline{Computation for I}: Using that $\tilde{\omg} \in \calA$, we may write \begin{equation*}
	\begin{split}
		I = \brk{\tilde{\omg} - \bar\omg, \bar{\psi} } =  \int \int_{r_2}^{r_2 + h_2} \bar\psi(r)  rdrd\tht  - \int \int_{r_1}^{r_1 + h_1} \bar\psi (r)  rdrd\tht  =: I_2 + I_1.
	\end{split}
\end{equation*}
Then, we compute using that $\bar\psi (r_2) = 0$ 
\begin{equation*}
	\begin{split}
		I_{2} &= \int \int_{r_2}^{r_2 + h_2}( \rd_r\bar\psi(r_2) (r-r_2) + o(r-r_2) ) (r_2 + r-r_2)drd\tht \\
		& =  r_2  \rd_r\bar\psi(r_2) \int \int_{r_2}^{r_2 + h_2}   (r-r_2)  drd\tht  +  o(\nrm{h_2}^2)  = \frac{r_2  \rd_r\bar\psi(r_2) }{2}  \int h_2^2 d\tht + o(\nrm{h_2}^2) .
	\end{split}
\end{equation*} Similarly, we have that \begin{equation*}
	\begin{split}
		I_1 =  -\frac{r_1  \rd_r\bar\psi(r_1) }{2}  \int h_1^2 d\tht + o(\nrm{h_1}^2) .
	\end{split}
\end{equation*} Note the negative sign in the first term of the right hand side. Therefore, we have that \begin{equation*}
	\begin{split}
		I \le -2c_0 \nrm{h}^2 + o(\nrm{h}^2)
	\end{split}
\end{equation*} for some $c_0>0$ depending only on $r_1$ and $r_2$.

\medskip

\noindent \underline{Computation for II}: Next, we consider the quadratic expression $II$ in polar coordinates after writing \begin{equation*}
	\begin{split}
		\tilde\omg - \bar\omg = (\tilde\omg_{h_2} - \bar\omg_{r_2}) - (\tilde\omg_{h_1} - \bar\omg_{r_1}), \qquad G_{j} := G(\tilde\omg_{h_j} - \bar\omg_{r_j});
	\end{split}
\end{equation*} \begin{equation*}
	\begin{split}
		II & = \frac12 \iint  (\tilde\omg_{h_2} - \bar\omg_{r_2}) (r,\tht) G_2 (r,\tht)  rdrd\tht  - \frac12 \iint  (\tilde\omg_{h_1} - \bar\omg_{r_1}) (r,\tht) G_2 (r,\tht)  rdrd\tht \\
		&\qquad - \frac12 \iint  (\tilde\omg_{h_2} - \bar\omg_{r_2}) (r,\tht) G_1 (r,\tht)  rdrd\tht + \frac12 \iint  (\tilde\omg_{h_1} - \bar\omg_{r_1}) (r,\tht) G_1 (r,\tht)  rdrd\tht \\
		&  =: II_{22} + II_{21} + II_{12} + II_{11}.
	\end{split}
\end{equation*} 
{As in \cite{Tang},}
 we see that \begin{equation*}
	\begin{split}
		II_{jj} = (r_j)^2 \brk{h_j, Kh_j} + o(\nrm{h_j}^2) 
	\end{split}
\end{equation*} for $j = 1, 2$ and \begin{equation*}
	\begin{split}
		II_{12}=II_{21} = - r_1r_2 \brk{ h_1, \tilde{K} h_2 } + o(\nrm{h}^2). 
	\end{split}
\end{equation*} Here, $\tilde{K}$ is the convolution operator defined on $S^1$ by \begin{equation*}
	\begin{split}
		(\tilde{K} h_2)(\tht) = \frac{1}{2\pi}  \int \ln \frac{1}{|r_1 e^{i\tht} - r_2 e^{i\tht'} |} h_2(\tht') d\tht' . 
	\end{split}
\end{equation*}
Therefore, we have that \begin{equation*}
	\begin{split}
		II = r_2^2 \brk{ h_2, K h_2 } + r_1^2 \brk{ h_1, K h_1} - 2r_1r_2 \brk{ h_1, \tilde{K} h_2} + o(\nrm{h}^2). 
	\end{split}
\end{equation*} Under the small mass condition \eqref{eq:small-mass}, we may replace $\tilde{K} h_2$ with the convolution   \begin{equation*}
	\begin{split}
		\frac{1}{2\pi}  \int \ln \frac{1}{|r_1  e^{i(\tht-\tht')} /r_2- 1|} h_2(\tht') d\tht' ,
	\end{split}
\end{equation*} whose eigenfunctions are simply $e^{in\tht}$ for $n\in\bbZ$. The eigenvalues of this operator depend on $r_1$ and $r_2$ but decays to 0 as $|n|\to\infty$, just like those for $K$. Therefore, we deduce that for $\tilde{\omg} \in \calA_m[\bar\omg]$, \begin{equation*}
	\begin{split}
		\left| II \right| \le o_m(1)\nrm{h}^2. 
	\end{split}
\end{equation*}

\medskip

\noindent \underline{Conclusion.} We have that \begin{equation*}
	\begin{split}
		I + II \le -c_0\nrm{h}^2 ,
	\end{split}
\end{equation*} for $m$ sufficiently large and $\nrm{h}$ sufficiently small. This concludes the proof of Proposition \ref{prop:ann}. \qedsymbol

\bibliography{bib_Kelvin-final_0518}
\bibliographystyle{plain}

\end{document}